\documentclass{article}

\usepackage{graphicx} 
\usepackage{amsfonts}
\usepackage{amssymb} 
\usepackage{latexsym}
\usepackage{pifont}
\usepackage{amssymb,amsfonts}
\usepackage{amsthm}
\usepackage{amsmath}
\usepackage{graphics,graphicx,psfrag}
\usepackage{amscd}
\usepackage{epsfig}
\usepackage[all]{xy}
\usepackage{float}
\usepackage[margin=1.25in]{geometry}

\newtheorem{theorem}{Theorem}[section]
\newtheorem{proposition}{Proposition}[section]
\newtheorem{definition}{Definition}[section]
\newtheorem{remark}{Remark}[section]
\newtheorem{lemma}{Lemma}[section]
\newtheorem{example}{Example}[section]

\newtheorem{corollary}{Corollary}[section]

\usepackage[hang,flushmargin]{footmisc} 
\usepackage{framed}
\usepackage[usenames,dvipsnames]{color}
\usepackage{tikz}
\usepackage{xcolor}
\usepackage{parskip}
\usetikzlibrary{matrix,arrows}
\usetikzlibrary{decorations.pathmorphing}
\usepackage{enumitem}
\usepackage{setspace}
\usepackage[affil-it]{authblk}
\setlist{nolistsep}

\usepackage{setspace}
\newcommand{\Modr}{ {\rm{\bf Mod}}_R }
\newcommand{\Modl}{ \mbox{}_R{\rm{\bf Mod}} }

\newcommand{\Modgrad}{ \mbox{}_A{\rm{\bf Mod}} }
\newcommand{\Cadr}{ {\rm {\bf Ch}}({\rm {\bf Mod}}_R) }
\newcommand{\Cadl}{ {\rm {\bf Ch}}(\mbox{}_R{\rm {\bf Mod}}) }
\newcommand{\Complexes}{ {\rm {\bf Ch}}(\mathcal{C}) }   

\begin{document}

\title{Gorenstein homological dimensions \\ and Abelian model structures}

\author{Marco A. P\'erez B.}

\affil{{\rm D\'epartement de Math\'ematiques.} \\ {\rm Universit\'e du Qu\'ebec \`a Montr\'eal.} \\ {\rm marco.perez$@$cirget.ca}}

\maketitle


\begin{abstract}
\noindent We construct new complete cotorsion pairs in the categories of modules and chain complexes over a Gorenstein ring $R$, from the notions of Gorenstein homological dimensions, in order to obtain new Abelian model structures on both categories. If $r$ is a positive integer, we show that the class of modules with Gorenstein-projective (or Gorenstein-flat) dimension $\leq r$ forms the left half of a complete cotorsion pair. Analogous results also hold for chain complexes over $R$. In any Gorenstein category, we prove that the class of objects with Gorenstein-injective dimension $\leq r$ is the right half of a complete cotorsion pair. The method we use in each case consists in constructing a cogenerating set for each pair. Later on, we give some applications of these results. First, as an extension of some results by M. Hovey and J. Gillespie, we establish a bijective correspondence between the class of differential graded $r$-projective complexes and the class of modules over $R[x] / (x^2)$ with Gorenstein-projective dimension $\leq r$, provided $R$ is left and right Noetherian with finite global dimension. The same correspondence is also valid for the (Gorenstein-)injective and (Gorenstein-)flat dimensions. 
\end{abstract}  


\section{Introduction}

One of the topics studied in homological algebra is the notion of projective, injective and flat modules over $R$, where $R$ is an associative ring with identity. One interesting property is the existence of projective resolutions and injective coresolutions that allow us to define the projective and injective dimension of every module. In the particular case where $R$ is a Gorenstein ring, another type of homological algebra can be developed in $\Modl$, the category of left $R$-modules, via the notions of Gorenstein-projective and Gorenstein-injective modules. This context suits very well to appreciate the connection existing between model category structures and the theory of cotorsion pairs, which was described by M. Hovey. 

\vspace{0.5cm}

\begin{theorem}[M. Hovey, {\cite[Theorem 2.2]{Hovey}}] Let $\mathcal{C}$ be a bicomplete Abelian category. If $(\mathcal{A}, \mathcal{B} \cap \mathcal{D})$ and $(\mathcal{A} \cap \mathcal{D}, \mathcal{B})$ are two compatible cotorsion pairs, such that the class $\mathcal{D}$ is thick, then there exists a unique Abelian model structure on $\mathcal{C}$ such that $\mathcal{A}$ is the class of cofibrant objects, $\mathcal{B}$ is the class of fibrant objects, and $\mathcal{D}$ is the class of trivial objects. 

Conversely, if $\mathcal{C}$ is equipped with an Abelian model structure where $\mathcal{C}_{of}$, $\mathcal{F}_{ib}$ and $\mathcal{T}$ denote the classes of cofibrant, fibrant and trivial objects, respectively, then $(\mathcal{C}_{of} \cap \mathcal{T}, \mathcal{F}_{ib})$ and $(\mathcal{C}_{of}, \mathcal{F}_{ib} \cap \mathcal{T})$ are two compatible and complete cotorsion pairs. 
 \end{theorem}
 
 \vspace{0.5cm}
 
As an application of this result, to which we shall refer as Hovey's correspondence, the same Hovey proved that the class of Gorenstein-projective modules forms the left half of a complete cotorsion pair (see \cite[Theorem 8.3]{Hovey}), in order to obtain a unique Abelian model structure on $\Modl$ where such modules form the class of cofibrant objects, and the trivial objects are given by the modules with finite projective dimension. Hovey also proved the existence of another model structure with the same trivial objects, and where a module is fibrant if, and only if, it is Gorenstein-injective (See \cite[Theorem 8.6]{Hovey}). We shall refer to these model structures as \underline{Gorenstein-projective} and \underline{Gorenstein-injective} model structures, respectively. 

One of the goals of this work is to extend Hovey's results to any Gorenstein-homological dimension and also to the category of chain complexes over $R$, applying Hovey's correspondence. Specifically, if $R$ is an $n$-Gorenstein ring and if $0 < r \leq n$ is an integer, then we show that the class $\mathcal{GP}_r$ of modules with Gorenstein-projective dimension at most $r$ is the left half of a cotorsion pair $(\mathcal{GP}_r, (\mathcal{GP}_r)^\perp)$ cogenerated by a set, and so complete. On the other hand, the class $\mathcal{P}_r$ of modules with projective dimension bounded by $r$ is also the left half of another cotorsion pair $(\mathcal{P}_r, (\mathcal{P}_r)^\perp)$ cogenerated by a set. These two pairs turn out to be compatible, and so we obtain the model structure described in Theorem \ref{Gorrprojmodel}. We also show that the same model structure exists on the category $\Cadl$ of chain complexes. 

Concerning the dual notion of Gorenstein-injective dimension, we shall show that the class $\mathcal{GI}_r(\mathcal{C})$ of objects in a locally Noetherian Gorenstein category $\mathcal{C}$ with Gorenstein-injective dimension $\leq r$, is the right half of a complete cotorsion pair $(\mbox{}^\perp(\mathcal{GI}_r(\mathcal{C})), \mathcal{GI}_r(\mathcal{C}))$. For the case $r = 0$ and $\mathcal{C} = \Modl$, with $R$ a Gorenstein ring, Hovey proved that the set of all $S \in \Omega^i(J)$, with $i \geq 0$ and $J$ running over the set of indecomposable injective modules, is a cogenerating set of $(\mbox{}^\perp(\mathcal{GI}_0), \mathcal{GI}_0)$. We shall show that the objects $S \in \Omega^i(J)$, with $i \geq r$ and $J$ running over the set of indecomposable injective objects of $\mathcal{C}$, provide a cogenerating set for $(\mbox{}^\perp(\mathcal{GI}_r(\mathcal{C})), \mathcal{GI}_r(\mathcal{C}))$. 

With respect to the Gorenstein-flat dimension, Hovey along with J. Gillespie found a new Abelian model structure on $\Modl$ where the cofibrant objects are given by the Gorenstein-flat modules, the fibrant objects by the cotorsion modules, and the trivial objects by the modules with finite flat dimension. In order to extend this model structure to the class $\mathcal{GF}_r$ of modules with Gorenstein-flat dimension bounded by $r$, we shall show that for every $M \in \mathcal{GF}_r$ it is possible to have a filtration by the set of all $N \in \mathcal{GF}_r$ such that ${\rm Card}(N) \leq \kappa$, where $\kappa$ is an infinite cardinal with $\kappa > {\rm Card}(R)$. During the process, we redefine the notion of pure sub-modules for the context of Gorenstein homological algebra. The latter result is motivated by the fact that every module in the class $\mathcal{F}_r$ of modules with flat dimension $\leq r$ has a filtration by the set of $N \in \mathcal{F}_r$ with cardinality $\leq \kappa$ (see \cite[Lemma 6.1]{Perez}). It shall follow that $(\mathcal{GF}_r, (\mathcal{GF}_r)^\perp)$ and $(\mathcal{F}_r, (\mathcal{F}_r)^\perp)$ are two complete and compatible cotorsion pairs, and hence Hovey's correspondence yields the model structure described in Theorem \ref{gorrflatmodel}. To extend this model structure to the category of chain complexes, we shall consider cotorsion pairs in $\Cadl$ where orthogonality is defined with respect to the functor $\overline{{\rm Ext}}^1_{\Cadl}(-,-)$, the first right derived functor of $\overline{{\rm Hom}}(-,-)$, the internal Hom in the closed monoidal category $(\Cadl, \overline{\otimes})$, where $\overline{\otimes}$ is the bar-tensor product considered by Gillespie in \cite{Gil} to obtain the flat model structure. 

There is a connection between the projective model structure on $\Cadl$ described in \cite[Section 2.3]{HoveyBook} and Hovey's Gorenstein-projective model structure on modules over the $\mathbb{Z}$-graded ring $R[x]/(x^2)$. Specifically, there is a bijective correspondence between the class of differential graded projective complexes (cofibrant objects of the former model structure) and the class of Gorenstein-projective $R[x] / (x^2)$-modules (cofibrant objects of the latter model structure), provided $R$ is a left and right Noetherian ring with finite global dimension (See \cite[Propositions 3.6, 3.8 and 3.10]{HoveyGillespie}). This correspondence also appears in the injective and flat cases. We extend this result to any homological dimension, i.e. we shall see (in Theorem \ref{perezcorresponde}) that the class of differential graded $r$-projective complexes (cofibrant objects of the $r$-projective model structure constructed in \cite{Perez}) is in bijective correspondence with the class of modules with Gorenstein-projective dimension at most $r$ (cofibrant objects of the model structure described in Theorem \ref{Gorrprojmodel}).


\section{Preliminaries}

In this section we present some preliminary notions and known facts in order to set some of the notation we shall use and the background we shall need. Since this work in influenced by Hovey's correspondence, we consider pertinent to recall first the terminology appearing in its statement. 

A \underline{model category} is given by a bicomplete category $\mathcal{C}$, together with three classes of maps in $\mathcal{C}$, called the cofibrations, fibrations and weak equivalences, satisfying the following axioms:
\begin{itemize}
\item[{\bf (1)}] Given a commutative diagram
\[ \begin{tikzpicture}
\matrix (m) [matrix of math nodes, row sep=2em, column sep=2em]
{ X & A \\ Y & B \\ };
\path[->]
(m-1-1) edge (m-1-2) edge node[left] {$f$} (m-2-1)
(m-1-2) edge node[right] {$g$} (m-2-2)
(m-2-1) edge (m-2-2);
\path[dotted,->]
(m-2-1) edge node[above,sloped] {$d$} (m-1-2);
\end{tikzpicture} \]
where $f$ is a cofibration, $g$ is a fibration, and where either $f$ or $g$ is a weak equivalence, then there exists a diagonal filler $d$ making the inner triangles commutative.  

\item[{\bf (2)}] Every map $h$ in $\mathcal{C}$ can be factored as $h = g \circ f = g' \circ f'$, where $f$ is cofibration and a weak equivalence, $f'$ is a cofibration, $g$ is a fibration, and $g'$ is a fibration and a weak equivalence. 

\item[{\bf (3)}] Let $f$ and $g$ be two maps in $\mathcal{C}$ such that the composition $g \circ f$ makes sense. If two out of three of the maps $f$, $g$ and $g \circ f$ are weak equivalences, then so is the third. 
\end{itemize}

Let $0$ and $1$ be the initial and terminal objects of $\mathcal{C}$, respectively. An object $X$ in $\mathcal{C}$ is said to be \underline{cofibrant} if the map $0 \longrightarrow X$ is a cofibration. \underline{Fibrant objects} are defined dually. Finally, $X$ is said to be \underline{trivial} if $0 \longrightarrow X$ is a weak equivalence.  

The book \cite{HoveyBook} is a good introduction to the theory of model categories. In this reference, factorizations in Axiom {\bf (2)} are assumed to be functorial. We shall consider functorial factorizations in this work, since models structures obtained from Hovey's correspondence satisfy that extra condition. 

The triplet formed by the above three classes of maps is known as a \underline{model} \underline{structure} on $\mathcal{C}$. We are interested in a special type of model structures on Abelian categories, known as \underline{Abelian model struc}- \underline{tures}, first defined by M. Hovey in \cite{Hovey}, and that satisfy the following conditions: 
\begin{itemize}
\item[ {\bf (1)}] $f$ is a cofibration (and a weak equivalence) if, and only if, $f$ is a monomorphism with cofibrant (and trivial) cokernel. 

\item[ {\bf (2)}] $f$ is a fibration (and a weak equivalence) if, and only if, $f$ is an epimorphism with fibrant (and trivial) kernel. 
\end{itemize} 

Two classes of objects $\mathcal{A}$ and $\mathcal{B}$ in an Abelian category $\mathcal{C}$ form a \underline{cotorsion pair} $(\mathcal{A,B})$ if  
\begin{align*}
\mathcal{A} & = \mbox{}^\perp\mathcal{B} := \{ A \in {\rm Ob}(\mathcal{C}) \mbox{ {\rm : }} {\rm Ext}^1_{\mathcal{C}}(A,B) = 0, \mbox{ for every }B \in \mathcal{B} \}, \mbox{ and } \\ 
\mathcal{B} & = \mathcal{A}^\perp := \{ B \in {\rm Ob}(\mathcal{C}) \mbox{ {\rm : }} {\rm Ext}^1_{\mathcal{C}}(A,B) = 0, \mbox{ for every }A \in \mathcal{A} \}.
\end{align*} 
\begin{itemize}
\item[$\bullet$] The pair $(\mathcal{A,B})$ is said to be \underline{complete} if for every object $X \in {\rm Ob}(\mathcal{C})$ there exist exact sequences $0 \longrightarrow B \longrightarrow A \longrightarrow X \longrightarrow 0$ and $0 \longrightarrow X \longrightarrow B' \longrightarrow A' \longrightarrow 0$ where $A, A' \in \mathcal{A}$ and $B, B' \in \mathcal{B}$. 

\item[$\bullet$] If $\mathcal{S} \subseteq \mathcal{A}$ is a set such that $\mathcal{B} = \mathcal{S}^\perp$, then the pair $(\mathcal{A,B})$ is said to be \underline{cogenerated} by $\mathcal{S}$. By the \underline{Eklof and Trlifaj Theorem}, every cotorsion pair cogenerated by a set is complete (See \cite[Theorem 10]{Eklof}\footnote{Although this result was originally proven in \cite[Theorem 10]{Eklof} for the category of modules, it is also valid in every Grothendieck category, as specified in \cite[Remark 3.2.2 (b)]{Gobel}.}). 

\item[$\bullet$] Let $M$ be a left $R$-module. We recall that a family $(M_\alpha )_{\alpha < \lambda}$ of sub-modules of $M$, indexed by some ordinal $\lambda$, is a \underline{continuous chain} of $M$ if: 
\begin{itemize}
\item[ {\bf (1)}] $M_0 = 0$ and $M = \bigcup_{\alpha < \lambda} M_\alpha$. 

\item[ {\bf (2)}] $M_\alpha \subseteq M_{\alpha'}$ whenever $\alpha \leq \alpha'$. 

\item[ {\bf (3)}] $M_\beta = \bigcup_{\alpha < \beta} M_\alpha$ for every limit ordinal $\beta < \lambda$. 
\end{itemize}

Let $(M_\alpha)_{\alpha < \lambda}$ be a continuous chain of $M$ and $\mathcal{S}$ be a set of left $R$-modules. Then $(M_\alpha)_{\alpha < \lambda}$ is said to be an \underline{$\mathcal{S}$-filtration} of $M$ if $M_{\alpha + 1} / M_\alpha$ is isomorphic to an element in $\mathcal{S}$ for every $\alpha + 1 < \lambda$. The module $M$ is said to be $\mathcal{S}$-filtered.\footnote{Continuous chains and filtrations of complexes are defined in the same way.} 

If $(\mathcal{A,B})$ is a cotorsion pair and $\mathcal{S} \subseteq \mathcal{A}$ is a set such that every object in $\mathcal{A}$ is $\mathcal{S}$-filtered, then $(\mathcal{A,B})$ is cogenerated by $\mathcal{S}$ (See \cite[Proposition 2.1]{Perez}). 
\end{itemize}

\vspace{0.5cm}

Two cotorsions pairs of the form $(\mathcal{A}, \mathcal{B} \cap \mathcal{D})$ and $(\mathcal{A} \cap \mathcal{D}, \mathcal{B})$ are called \underline{compatible}. 

A class of objects $\mathcal{A}$ in $\mathcal{C}$ is said to be \underline{thick} if it is closed under retracts, and if for every short exact sequence $0 \longrightarrow A' \longrightarrow A \longrightarrow A'' \longrightarrow 0$, if two out of three of the terms $A'$, $A$ and $A''$ are in $\mathcal{A}$, then so is the third. For example, the class $\mathcal{E}$ of exact chain complexes is thick, and so is the class $\mathcal{W}$ of $R$-modules with finite projective dimension, provided $R$ is a Gorenstein ring. 

We conclude this section recalling some constructions and functors in the category of chain complexes over a ring. Given two chain complexes $X \in \Cadr$ and $Y \in \Cadl$, the \underline{standard tensor product complex} $X \otimes Y$ is the chain complex (of Abelian groups) given at each $n$th term by $(X \otimes Y)_n := \bigoplus_{k \in \mathbb{Z}} X_k \otimes_R Y_{n-k}$, whose the boundary maps are defined by $\partial^{X \otimes Y}_n(x \otimes y) := \partial^X_k(x) \otimes y + (-1)^{k}x \otimes \partial^Y_{n-k}(y)$, for every $x \otimes y \in X_k \otimes Y_{n-k}$. This construction defines a functor $- \otimes -$, from which one constructs left derived functors ${\rm Tor}_i(-,-)$. The \underline{bar-tensor product} of $X$ and $Y$ is the chain complex $X \overline{\otimes} Y$ of Abelian groups given by $(X \overline{\otimes} Y)_n := (X \otimes Y)_n / B_n(X \otimes Y)$, for every $n \in \mathbb{Z}$, whose boundary maps are given by $\partial^{X \overline{\otimes} Y}_n(\overline{x \otimes y}) := \overline{\partial^X(x) \otimes y}$. We denote the left derived functors of $- \overline{\otimes} -$ by $\overline{{\rm Tor}}_i(-,-)$. 

Now let $X$ and $Y$ be two chain complexes in $\Cadl$. The complex ${\rm Hom}'(X, Y)$ in $\Cadl$ is defined by ${\rm Hom}'(X, Y)_n := \prod_{k \in \mathbb{Z}} {\rm Hom}_R(X_k, Y_{n+k})$, for every $n \in \mathbb{Z}$. The boundaries are given by $\partial^{{\rm Hom}'(X,Y)}_n(f) := (\partial^Y_{k+n} \circ f_k - (-1)^nf_{k-1} \circ \partial^X_k)_{k \in \mathbb{Z}}$, for every $f = (f_k)_{k \in \mathbb{Z}} \in {\rm Hom}'(X, Y)_n$. From ${\rm Hom}'(X,Y)$ one constructs the \underline{bar-Hom} complex $\overline{{\rm Hom}}(X,Y)$ by setting $\overline{{\rm Hom}}(X,Y)_n := Z_n({\rm Hom}'(X,Y))$, for every $n \in \mathbb{Z}$. The boundary maps are given by $\partial^{\overline{{\rm Hom}}(X,Y)}_n(f) := (\partial^Y_{k+n} \circ f_k)_{k \in \mathbb{Z}}$, for every $f = (f_k)_{k \in \mathbb{Z}} \in \overline{{\rm Hom}}(X,Y)_n$. Notice that for $n = 0$ we have $\partial^Y_k \circ f_k = f_{k-1} \circ \partial^X_k$, i.e. $f$ is a chain map and $\overline{{\rm Hom}}(X,Y)_0 = {\rm Hom}(X,Y)$. For $n = 1$, $\partial^Y_{k+1} \circ f_k + f_{k-1} \circ \partial^X_k = 0$, i.e. $f$ is a chain homotopy from $0$ to $0$. Every $f = (f_k)_{k \in \mathbb{Z}} \in Z_n({\rm Hom}'(X,Y))$ is known as a \underline{map of degree} $n$. We denote the right derived functors of $\overline{{\rm Hom}}(-,-)$ by $\overline{{\rm Ext}}^i(-,-)$. 

\vspace{0.5cm}

\begin{definition} {\rm Recall that if $M$ is a left $R$-module, the \underline{Pontryagin} or \underline{character module} of $M$ is the right $R$-module $M^+ = {\rm Hom}_{\mathbb{Z}}(M, \mathbb{Q / \mathbb{Z}})$. 

Given a chain complex $X = (X_m, \partial^X_m)_{m \in \mathbb{Z}}$ in $\Cadl$, the \underline{Pontryagin} or \underline{character complex} of $X$ is the complex $X^+$ in $\Cadr$ given by $(X^+)_m := {\rm Hom}_{\mathbb{Z}}(X_{-m-1}, \mathbb{Q / Z})$, where the boundary maps are defined by $\partial^{X^+}_m := (-1)^{m-1} {\rm Hom}_{\mathbb{Z}}(\partial^X_{-m}, \mathbb{Q / Z})$.\footnote{We differ from the definition given in \cite[Theorem 4.1.3 {\bf (4)}]{GR}, although we use the same notation $X^+$.}  }
\end{definition}

\vspace{0.5cm}

\begin{definition} {\rm Let $\mathcal{C}$ be an  Abelian category and $C \in {\rm Ob}(\mathcal{C})$. The \underline{$m$th sphere complex centred} \underline{at $C$}, denoted $S^m(C)$, is defined by $(S^m(C))_m := C$ and $(S^m(C))_k := 0$ if $k \neq m$. The \underline{$m$th disk} \underline{complex centred at $C$}, denoted $D^m(C)$, is defined by $(D^m(C))_k := C$ if $k = m$ or $m-1$, and $0$ otherwise, where the boundary map $\partial^{D^m(C)}_m$ is the identity ${\rm id}_C$. }
\end{definition}

\vspace{0.5cm}

\begin{proposition} $X^+ \cong \overline{{\rm Hom}}_{\mathbb{Z}}(X, D^0(\mathbb{Q / Z})$.
\end{proposition}
\begin{proof} First, notice that
\begin{align*}
{\rm Hom}'(X, D^0(\mathbb{Q / Z}))_m & = {\rm Hom}_{\mathbb{Z}}(X_{-m-1}, \mathbb{Q / Z}) \times {\rm Hom}_{\mathbb{Z}}(X_{-m}, \mathbb{Q / Z}), \\
{\rm Hom}'(X, D^0(\mathbb{Q / Z}))_{m-1} & = {\rm Hom}_{\mathbb{Z}}(X_{-m}, \mathbb{Q / Z}) \times {\rm Hom}_{\mathbb{Z}}(X_{-m+1}, \mathbb{Q / Z}).
\end{align*}
Every $f = (f_k)_{k \in \mathbb{Z}} \in {\rm Hom}(X, D^0(\mathbb{Q / Z}))_m$ has the form $f = (\cdots, 0, f_{-m-1}, f_{-m}, 0, \cdots)$. Now suppose $\partial^{\overline{{\rm Hom}}(X,Y)}_n (f) = 0$. Then we have $\partial^{D^0(\mathbb{Q / Z})}_{k+m} \circ f_k - (-1)^m f_{k-1} \circ \partial^X_k = 0$ for every $k \in \mathbb{Z}$. In particular, $0 = f_{-m} \circ \partial^X_{-m+1}$ and $f_{-m} = (-1)^m f_{-m-1} \circ \partial^X_{-m}$. Define a map \[ \varphi_m : \overline{{\rm Hom}}(X, D^0(\mathbb{Q / Z}))_m \longrightarrow (X^+)_m = {\rm Hom}_{\mathbb{Z}}(X_{-m-1}, \mathbb{Q / Z}) \] by setting $\varphi_m(f) := (-1)^m f_{-m-1}$ for every $m \in \mathbb{Z}$. Using the previous two equalities, it is easy to check that the family of maps $\varphi := (\varphi_m)_{m \in \mathbb{Z}}$ defines an isomorphism of chain complexes $\overline{{\rm Hom}}(X, D^0(\mathbb{Q / Z})) \longrightarrow X^+$.
\end{proof}

The following proposition can be proven by considering long exact sequences of $\overline{{\rm Ext}}^i(-,-)$ and $\overline{{\rm Tor}}_i(-,-)$, and applying the corresponding natural isomorphisms given in \cite[Proposition 4.2.1]{GR}.

\vspace{0.5cm}

\begin{proposition}\label{isomorfos} Let $X$ and $Y$ be two chain complexes over $\Modr$ and $\Modl$, respectively. For every $i > 0$, we have the following isomorphisms of complexes:
\begin{itemize}
\item[ {\bf (1)}] $\overline{\rm Ext}^i(X, Y^+) \cong \overline{{\rm Tor}}_i(X, Y)^+$.

\item[ {\bf (2)}] $\overline{{\rm Tor}}_i(X, Y) \cong \overline{{\rm Tor}}_i(X, Y)$, provided $R$ is a commutative ring.  
\end{itemize} 
\end{proposition}


\section{Gorenstein-projective and Gorenstein-injective objects in Gorenstein categories}

In this section we study the concepts of Gorenstein-projective and Gorenstein-injective objects in Gorenstein categories, which are a sort of generalization of the category of modules over a Gorenstein ring. We start recalling the notions of left and right dimensions with respect to a class of objects in an Abelian category. 

\vspace{0.5cm}

\begin{definition} {\rm Given a class of objects $\mathcal{F}$ of an Abelian category $\mathcal{C}$, we shall say that an object $X$ is a \underline{left $n$-$\mathcal{F}$-object} if there is an exact sequence \[ 0 \longrightarrow F_n \longrightarrow F_{n-1} \longrightarrow \cdots \longrightarrow F_1 \longrightarrow F_0 \longrightarrow X \longrightarrow 0, \] where $F_i \in \mathcal{F}$ for every $0 \leq i \leq n$. If $n$ is the smallest integer for which such a sequence exists, then we say that $X$ has \underline{left $\mathcal{F}$-dimension} equal to $n$. If such an integer $n$ does not exist, we shall say instead that $X$ has infinite left $\mathcal{F}$-dimension. Dually, one defines \underline{right $n$-$\mathcal{F}$-objects} and \underline{right $\mathcal{F}$-dimensions}. }
\end{definition} 

\newpage

Let $\mathcal{P}_n(\mathcal{C})$ and $\mathcal{I}_n(\mathcal{C})$ denote the classes of (left) $n$-projective and (right) $n$-injective objects of $\mathcal{C}$, respectively. For the category $\Modl$, $\mathcal{F}_n$ shall denote the class of $n$-flat  modules. For simplicity, we shall write $\mathcal{P}_n(\Modl) = \mathcal{P}_n$ and $\mathcal{I}_n(\Modl) = \mathcal{I}_n$. We shall denote by ${\rm pd}(X)$, ${\rm id}(X)$ and ${\rm fd}(X)$ the projective, injective and flat dimension of an object $X$ (say a module or a chain complex), respectively. \\

\begin{definition}[See {\cite[Definition 2.18]{Estrada}}] \label{defGor} {\rm A Grothendieck category $\mathcal{C}$ is said to be a \underline{Gorenstein cate}- \underline{gory} if the following conditions are satisfied: 
\begin{itemize}
\item[ {\bf (1)}] For any object $X$ of $\mathcal{C}$, ${\rm pd}(X) < \infty$ if, and only if, ${\rm id}(X) < \infty$. 

\item[ {\bf (2)}] $FDP(\mathcal{C}) := {\rm sup}\{ {\rm pd}(X) \mbox{}:\mbox{} {\rm pd}(X) < \infty \} < \infty$ and $FDI(\mathcal{C}) := {\rm sup}\{ {\rm id}(X) \mbox{}:\mbox{} {\rm id}(X) < \infty \} < \infty$. 

\item[ {\bf (3)}] $\mathcal{C}$ has a generator of finite projective dimension. 
\end{itemize} }
\end{definition}

Let $\mathcal{C}$ be a Gorenstein category and $\mathcal{W}(\mathcal{C})$ denote the class of objects of $\mathcal{C}$ with finite projective dimension. Notice that $\mathcal{W}(\mathcal{C})$ is thick. In the case $\mathcal{C} = \Modl$, we shall write $\mathcal{W}(\Modl) = \mathcal{W}$. \\

\begin{remark} {\rm Every Gorenstein category has enough projective objects. For let $X \in {\rm Ob}(\mathcal{C})$ and $G$ be a generator of $\mathcal{C}$ of finite projective dimension, say $n$. Then by {\rm \cite[Lemma 2, Section 2.10]{Pareigis}} there exists an epimorphism $G^{(I)} \twoheadrightarrow X$, where $G^{(I)}$ denotes the direct sum of copies of $G$ over the set $I = {\rm Hom}_{\mathcal{C}}(G,X)$. On the other hand, there exists a projective object $P$ and an epimorphism $P \twoheadrightarrow G$. Then we get an epimorphism $P^{(I)} \twoheadrightarrow X$, where $P^{(I)}$ is a projective object. It follows by {\rm \cite[Lemma 2, Section 2.10]{Pareigis}} again, that $P$ is a projective generator of $\mathcal{C}$, so condition {\bf (3)} in the previous definition can be replaced by {\bf (3')}: $\mathcal{C}$ has a projective generator. }
\end{remark} 

\vspace{0.5cm}

We recall that a ring $R$ (not necessarily commutative) is called a \underline{Gorenstein ring} if it is both left and right Noetherian, and it has finite injective dimension as either a left or right $R$-module. It can be shown that over such a ring, the injective dimensions of $R$ as a left and right $R$-module coincide to some nonnegative integer, say $n$. Moreover, if $R$ is a Gorenstein ring, then conditions ${\rm pd}(M) \leq n$, ${\rm pd}(M) < \infty$, ${\rm id}(M) \leq n$, ${\rm id}(M) < \infty$, ${\rm fd}(M) \leq n$ and ${\rm fd}(M) < \infty$ are equivalent, for every (left or right) $R$-module $M$. This fact was proven by Y. Iwanaga, and it is after him that Gorenstein rings are also known as \underline{$n$-Iwanaga-Gorenstein rings}. The reader can see the details in \cite[Theorem 9.1.10]{EJ}. It follows by the previous comments that if $R$ is a Gorenstein ring, then $\Modl$ is a Gorenstein category. We give two more examples of Gorenstein categories below. 

\newpage

\begin{example} {\rm \
\begin{itemize}
\item[{\bf (1)}] Recall that a ring $R$ is called a \underline{quasi-Frobenius ring} if the classes of projective and injective left (or right) $R$-modules coincide. Every quasi-Frobenius ring $R$ is a $0$-Gorenstein ring, and so $\Modl$ is a Gorenstein category. 

\item[{\bf (2)}] Consider the category $\Modl$ where $R$ is a field $\mathbb{K}$. So $\Modl = {\rm {\bf Vect}}_{\mathbb{K}}$ is the category of $\mathbb{K}$-vector spaces. Since every vector space is projective and injective, $\mathbb{K}$ is a quasi-Frobenius ring and ${\rm {\bf Vect}}_{\mathbb{K}}$ is a Gorenstein category. 
\end{itemize} }
\end{example}

\vspace{0.25cm}

Before giving another example of a Gorenstein category, we recall from \cite[Definition 3.3]{Gil} that if $\mathcal{A}$ is a class of objects in an abelian category $\mathcal{C}$, then a complex $X$ in $\Complexes$ (the category of chain complexes over $\mathcal{C}$) is said to be an \underline{$\mathcal{A}$-complex} if $X$ is exact and $Z_m(X) \in \mathcal{A}$ for every $m \in \mathbb{Z}$. The class of $\mathcal{A}$-complexes is denoted by $\widetilde{\mathcal{A}}$.  In \cite[Proposition 3.1]{Perez}, it is proven that $\widetilde{\mathcal{I}_n(\mathcal{C})} = \mathcal{I}_n(\Complexes)$ if $\mathcal{C}$ has enough injectives. Dually, $\widetilde{\mathcal{P}_n(\mathcal{C})} = \mathcal{P}_n(\Complexes)$ if $\mathcal{C}$ has enough projectives. Moreover, $\widetilde{\mathcal{F}_n}$ is the class of (left) $n$-flat complexes. 

\vspace{0.5cm}

\begin{example} {\rm If $\mathcal{C}$ is a Gorenstein category, then so is $\Complexes$. For conditions {\bf (1)} and {\bf (2)} of Definition \ref{defGor}, it suffices to notice that if $X$ is a complex with finite projective dimension, then $X \in \widetilde{\mathcal{P}_n(\mathcal{C})}$, for $n = FDP(\mathcal{C})$. To check condition {\bf (3')}, if $G$ is a projective generator of $\mathcal{C}$, then notice that $\bigoplus_{m \in \mathbb{Z}} D^m(G)$ is a projective generator of $\Complexes$. In particular, ${\rm Ch}(\Modl)$ is a Gorenstein category if $R$ is a Gorenstein ring. }
\end{example}

\vspace{0.25cm}

Now we recall the definition of Gorenstein-projective and Gorenstein-injective objects. Later in Section 5, we shall recall the notions of Gorenstein-flat modules and complexes. 

Given an Abelian category $\mathcal{C}$ and a class $\mathcal{F}$ of objects of $\mathcal{C}$, we recall from \cite[Definition 8.1.1]{EJ} that a chain complex $X$ over $\mathcal{C}$ is \underline{${\rm Hom}_{\mathcal{C}}(\mathcal{F},-)$-exact} if the complex \[ {\rm Hom}_{\mathcal{C}}(F,X) = \cdots \longrightarrow {\rm Hom}_{\mathcal{C}}(F,X_1) \longrightarrow {\rm Hom}_{\mathcal{C}}(F,X_0) \longrightarrow {\rm Hom}_{\mathcal{C}}(F,X_{-1}) \longrightarrow \cdots \] is exact, for every $F \in \mathcal{F}$. The notion of \underline{${\rm Hom}_{\mathcal{C}}(-,\mathcal{F})$-exact} complexes is dual. 

\vspace{0.5cm}

\begin{definition}[See {\cite[Definition 2.19]{Estrada}}] {\rm An object $X$ of an Abelian category $\mathcal{C}$ is \underline{Gorenstein}-\underline{projective} if there exists an exact sequence $\cdots \longrightarrow P_1 \longrightarrow P_0 \longrightarrow P^0 \longrightarrow P^1 \longrightarrow \cdots$ of projective objects which is ${\rm Hom}_{\mathcal{C}}(-, \mathcal{P}_0(\mathcal{C}))$-exact and such that $X = {\rm Ker}(P^0 \longrightarrow P^1)$. Dually, $X$ is \underline{Gorenstein-injective} if there exists an exact sequence $\cdots \longrightarrow I_1 \longrightarrow I_0 \longrightarrow I^0 \longrightarrow I^1 \longrightarrow \cdots$ of injective objects which is ${\rm Hom}_{\mathcal{C}}(\mathcal{I}_0(\mathcal{C}),-)$-exact and such that $X = {\rm Ker}(I^0 \longrightarrow I^1)$. We denote by $\mathcal{GP}_0(\mathcal{C})$ and $\mathcal{GI}_0(\mathcal{C})$ the classes of  Gorenstein-projective and Gorenstein-injective objects, respectively. }
\end{definition}

\vspace{0.5cm}

In \cite[Theorems 2.24 and 2.25]{Estrada}, it is proven that $\mathcal{W}(\mathcal{C})$ is the right and left half of two complete cotorsion pairs $(\mathcal{GP}_0(\mathcal{C}), \mathcal{W}(\mathcal{C}))$ and $(\mathcal{W}(\mathcal{C}), \mathcal{GI}_0(\mathcal{C}))$ in a Gorenstein category $\mathcal{C}$. The proofs given there consist in constructing for every $X \in {\rm Ob}(\mathcal{C})$ an exact sequence $0 \longrightarrow W \longrightarrow C \longrightarrow X \longrightarrow 0$ with $C \in \mathcal{GP}_0(\mathcal{C})$ and $W \in \mathcal{W}(\mathcal{C})$, and in showing that $(\mathcal{W}, \mathcal{GI}_0(\mathcal{C}))$ is a small cotorsion pair. Moreover, cogenerating sets for these pairs were provided by M. Hovey in the case $\mathcal{C} = \Modl$, with $R$ an $n$-Gorenstein ring. On the one hand, if $\mathcal{T}$ denotes the set of all $n$-syzygies in $\Omega^n(R / I)$, where $I$ runs over the set of left ideals of $R$, then $\mathcal{T}$ cogenerates $(\mathcal{GP}_0,\mathcal{W})$ (See \cite[Theorem 8.3]{Hovey}). On the other hand, if $\mathcal{S}$ denotes the set of all modules $S \in \Omega^i(J)$, where $i \geq 0$ and $J$ runs over the set of indecomposable injective modules, then $\mathcal{S}$ cogenerates $(\mathcal{W}, \mathcal{GI}_0)$ (See \cite[Theorem 8.4]{Hovey}). These two results can be stated for any Gorenstein category $\mathcal{C}$ satisfying certain conditions. Let $G$ be a (projective) generator of $G$. By \cite[Lemme 1, page 136]{Grothendieck}, an object $Y$ in $\mathcal{C}$ is $m$-injective if, and only if, ${\rm Ext}^{m+1}_{\mathcal{C}}(G/J,Y) = 0$ for every sub-object $J$ of $G$. Notice also that there exists a nonnegative integer $N$ such that $\mathcal{W}(\mathcal{C}) = \mathcal{I}_N(\mathcal{C})$. On the other hand, $\mathcal{GP}_0(\mathcal{C}) = \mbox{}^\perp (\mathcal{W}(\mathcal{C}))$ by \cite[Theorem 2.25]{Estrada}. Hence $(\mathcal{GP}_0(\mathcal{C}), \mathcal{W}(\mathcal{C})) = (\mbox{}^\perp(\mathcal{I}_N(\mathcal{C})),\mathcal{I}_N(\mathcal{C}))$. Therefore, we obtain the following result. 

\vspace{0.5cm}

\begin{proposition}[M. Hovey, {\cite[Theorem 8.3]{Hovey}} for Gorenstein categories] If $\mathcal{C}$ is a Gorenstein category with a generator $G$, then the pair $(\mathcal{GP}_0(\mathcal{C}), \mathcal{W}(\mathcal{C}))$ is cogenerated by the set $\mathcal{T}$ of all $n$-syzygies in $\Omega^n(G/J)$, where $J$ runs over the set of sub-objects of $G$. 
\end{proposition}

\vspace{0.5cm}

\begin{remark} {\rm Let $R$ be an $n$-Gorenstein ring. By \cite[Proposition 3.1]{Perez}, notice that $\widetilde{\mathcal{W}} = \mathcal{W}({\rm {\bf Ch}}(\mathcal{C}))$. Denote $\widehat{\mathcal{GP}_0} = \mathcal{GP}_0(\Cadl)$. In {\rm \cite{GR}}, it is proven that $(\widehat{\mathcal{GP}_0}, \widetilde{W})$ is a complete cotorsion pair. This proof is direct in the sense that for every complex $X$ it is constructed a short exact sequence $0 \longrightarrow W \longrightarrow C \longrightarrow X \longrightarrow 0$ where $C$ is a Gorenstein-projective complex and $W \in \widetilde{\mathcal{W}}$. We can give another cogenerating set for this pair, namely \[ \mathcal{X} = \{ X \in \Omega^n(S^m(R / I)) \mbox{ : $m \in \mathbb{Z}$ and $I$ is a left ideal of $R$} \} \cup \{ \Sigma^k(S^0(R)) \mbox{ : } k \in \mathbb{Z} \}. \] The equality $\widetilde{\mathcal{W}} = \mathcal{X}^\perp$ follows by \cite[Lemma 4.2]{Gillespie} and \cite[Lemma 5.1]{Rada}. }
\end{remark}

\vspace{0.5cm}

To show that \cite[Theorem 8.4]{Hovey} is also valid for Gorenstein categories $\mathcal{C}$, we impose an additional condition on $\mathcal{C}$. Namely, we need $\mathcal{C}$ to be \underline{locally Noetherian}, i.e. $\mathcal{C}$ has a Noetherian generator $N$. Since $\mathcal{C}$ is a Grothendieck category, the class ${\rm sub}(N)$ of sub-objects of $N$ is actually a set by \cite[Lemma 1, page 111]{Pareigis}. So $N$ is \underline{Noetherian} if each nonempty subset $\mathcal{U}$ of ${\rm sub}(N)$ has a maximal sub-object, where a sub-object $Y \in \mathcal{U}$ is said to be \underline{maximal} in $\mathcal{U}$ if $Y' \in \mathcal{U}$ and $Y \subseteq Y'$ always imply $Y = Y'$. 

An injective object $X$ in a Grothendieck category $\mathcal{C}$ is said to be \underline{indecomposable} if $X \neq 0$ and if, for each decomposition $X = X_1 \oplus X_2$ into a direct sum of injective objects, either $X = X_1$ or $X = X_2$. If $X$ is not indecomposable, then it is said to be \underline{decomposable}. 

\vspace{0.5cm}

\begin{remark} {\rm A theorem due to E. Matlis (see \cite[Theorem 4, Section 4.10]{Pareigis}) states that if $\mathcal{C}$ is a locally Noetherian Grothendieck category, then each injective object $I$ in $\mathcal{C}$ may be decomposed into a coproduct of indecomposable injective objects $I = \bigoplus_{\alpha \in A} I_\alpha$. If $R$ is a Noetherian ring (and so $\Modl$ is locally Noetherian), we can obtain decompositions as above of injective chain complexes in terms of indecomposable injective left $R$-modules. Let $I$ be an injective chain complex. Then we can write $I \cong \bigoplus_{m \in \mathbb{Z}} D^{m+1}(Z_{m}(I))$. For every $m \in \mathbb{Z}$, $Z_{m}(I) \cong \bigoplus_{\alpha_m \in \Lambda_m} J_{\alpha_m}$, where $J_{\alpha_m}$ is an indecomposable injective module. Hence $D^{m+1}(Z_m(I)) \cong \bigoplus_{\alpha_m \in \Lambda_m} D^{m+1}(J_{\alpha_m})$, and so we have 
\begin{align*}
I & \cong \bigoplus_{m \in \mathbb{Z}} \left( \bigoplus_{\alpha_m \in \Lambda_m} D^{m+1}(J_{\alpha_m}) \right) = \bigoplus \left\{ D^{m+1}(J_{\alpha_m}) \mbox{ : } (\alpha_{m+k})_{k \in \mathbb{Z}} \in \bigcup_{m \in \mathbb{Z}} \left( \prod_{k \in \mathbb{Z}} \Lambda_{m+k} \right) \right\}, 
\end{align*}
where each $D^{m+1}(J_{\alpha_m})$ is an indecomposable injective complex. Notice that a complex $J$ is an indecomposable injective if, and only if, $J$ is the disk complex of an indecomposable injective module. 

In a similar way, one can show that these results hold for the category of complexes over a locally Noetherian Grothendieck category. }
\end{remark}

\vspace{0.5cm}

In the proof of Theorem \cite[Theorem 8.4]{Hovey}, M. Hovey uses a corollary of the Eklof and Trlijaf Theorem given in \cite[Corollary 3.2.3]{Gobel}. The proof appearing there carry over to any Grothendieck category. Thanks to this fact and the equality $\mathcal{GI}_0(\mathcal{C}) = (\mathcal{W}(\mathcal{C}))^\perp$ proven in \cite[Theorem 2.24]{Estrada}, Hovey's arguments can be adapted to the context of locally Noetherian Gorenstein categories. 

\vspace{0.5cm}

\begin{proposition}[M. Hovey, {\cite[Theorem 8.4]{Hovey}} for Gorenstein categories]\label{hoveygorinjcat} Let $\mathcal{C}$ be a locally Noetherian Gorenstein category. Let $\mathcal{S}$ denote the set of all objects $S \in \Omega^i(J)$, where $i \geq 0$ and $J$ runs over the set of indecomposable injective objects of $\mathcal{C}$. Then $\mathcal{S}$ cogenerates a cotorsion pair $(\mathcal{W}(\mathcal{C}), \mathcal{GI}_0(\mathcal{C}))$. 
\end{proposition}

\vspace{0.5cm}

In every Abelian category $\mathcal{C}$ with enough projective and injective objects, the cotorsion pairs $(\mathcal{P}_0(\mathcal{C}), \mathcal{C})$ and $(\mathcal{C}, \mathcal{I}_0(\mathcal{C}))$ are complete. This occurs in any Gorenstein category. Moreover, the equalities $\mathcal{GP}_0 \cap \mathcal{W} = \mathcal{P}_0$ and $\mathcal{GI}_0 \cap \mathcal{W} = \mathcal{I}_0$ proven by Hovey in \cite[Corollary 8.5]{Hovey} also hold for such a category. So by Hovey's correspondence and the previous results, we have the following model structures.

\vspace{0.5cm}

\begin{theorem}[M. Hovey] Let $\mathcal{C}$ be a Gorenstein category. There is a unique Abelian model structure on $\mathcal{C}$, where $\mathcal{GP}_0(\mathcal{C})$, ${\rm Ob}(\mathcal{C})$ and $\mathcal{W}(\mathcal{C})$ are the classes of cofibrant, fibrant, and trivial objects, respectively. Dually, there is a unique Abelian model structure on $\mathcal{C}$, with the same trivial objects, such that ${\rm Ob}(\mathcal{C})$ and $\mathcal{GI}_0(\mathcal{C})$ are the classes of cofibrant and fibrant objects, respectively. 
\end{theorem}

\vspace{0.5cm}

We know that ${\rm {\bf Ch}}(\mathcal{C})$ is a Gorenstein category if $\mathcal{C}$ is. The next result provides a characterization of the Gorenstein-projective and Gorenstein-injective chain complexes. This is also proven by Garc\'ia Rozas in \cite[Theorem 3.3.5 \& Corollary 3.2.3]{GR} for complexes over Gorenstein rings. In the author's opinion, the proof given next is shorter. 

\vspace{0.5cm}

\begin{proposition}\label{Gorcomplexes} Let $\mathcal{C}$ be a Gorenstein category. 
\begin{itemize}
\item[ {\bf (1)}] A chain complex $X$ over $\mathcal{C}$ is Gorenstein-projective, if and only if, $X_m$ is a Gorenstein-projective object of $\mathcal{C}$, for every $m \in \mathbb{Z}$. 

\item[ {\bf (2)}] A chain complex $Y$ over $\mathcal{C}$ is Gorenstein-injective if, and only if, $Y_m$ is a Gorenstein-injective object of $\mathcal{C}$, for every $m \in \mathbb{Z}$. 
\end{itemize}
\end{proposition}
\begin{proof} We only prove the Gorenstein-projective case. Let $X$ be a Gorenstein-projective complex and $W \in \mathcal{W}(\mathcal{C})$. Then $D^{m+1}(W) \in \widetilde{\mathcal{W}(\mathcal{C})}$, for every $m \in \mathbb{Z}$. We have $0 = {\rm Ext}^1(X, D^{m+1}(W)) \cong {\rm Ext}^1(X_m, W)$ by \cite[Lemma 3.1]{Gil}. Hence $X_m \in \mbox{}^\perp\mathcal{W}(\mathcal{C}) = \mathcal{GP}_0(\mathcal{C})$. For the converse implication, if $X$ is a chain complex such that each $X_m$ is a Gorenstein-projective object, it suffices to show that ${\rm Ext}^{i+1}(X, P) = 0$, for every projective complex $P$ and every $i \geq 0$. Since $P$ is exact and each $Z_m(P)$ is projective in $\mathcal{C}$, we can write $P \cong \bigoplus_{m \in \mathbb{Z}} D^{m+1}(Z_m(P))$. Notice that in this case we actually have $P \cong \prod_{m \in \mathbb{Z}} D^{m+1}(Z_m(P))$. By \cite[Lemma 3.1]{Gil}, we have ${\rm Ext}^{i+1}(X, P) = 0$. Hence, $X$ is a Gorenstein-projective complex. 
\end{proof}

\vspace{0.25cm}

We finish this section checking if Hovey's Gorenstein-projective and Gorenstein-injective model structures are monoidal in the cases $\mathcal{C} = \Modl$ and $\mathcal{C} = \Cadl$, with $R$ a commutative Gorenstein ring. Both $\Modl$ and $\Cadl$ are equipped with monoidal structures given by the tensor products $\otimes_R$ and $\otimes$, respectively. Since $R$ is commutative, then these structures are symmetric. The unit object $S$ in $(\Cadl, \otimes)$ is given by $S = S^0(R)$. (See \cite[Section 4.1]{HoveyBook} for details). 

A \underline{monoidal model category} is a model category $\mathcal{C}$ equipped with a symmetric monoidal structure $(\otimes, S)$ such that the following conditions are satisfied: 
\begin{itemize}
\item[{\bf (1)}] For every pair of maps $f : U \longrightarrow V$ and $g : W \longrightarrow X$, the pushout of $f \otimes {\rm id}_W$ and ${\rm id}_U \otimes g$ induces a map $f \square g : (V \otimes W) \coprod_{U \otimes W} (U \otimes X) \longrightarrow V \otimes X$ making the following diagram commute: 
\[ \begin{tikzpicture}
\matrix (m) [matrix of math nodes, row sep=2em, column sep=2em]
{ U \otimes W & U \otimes X \\ V \otimes W & (V \otimes W) \coprod_{U \otimes W} (U \otimes X) \\ & & V \otimes X \\  };
\path[->]
(m-1-1) edge node[above] {${\rm id}_U \otimes g$} (m-1-2) edge node[left] {$f \otimes {\rm id}_W$} (m-2-1) (m-1-2) edge (m-2-2) (m-2-1) edge (m-2-2)
(m-1-2) edge [bend left=30] (m-3-3) (m-2-1) edge [bend right=30] (m-3-3);
\path[dotted,->]
(m-2-2) edge node[sloped] {$f \square g$} (m-3-3);
\end{tikzpicture} \]
If, given cofibrations $f : U \longrightarrow V$ and $g : W \longrightarrow X$ in $\mathcal{C}$, the induced map $f \square g$ is a cofibration, which is trivial if either $f$ or $g$ is. 

\item[{\bf (2)}] Using functorial factorizations, write $0 \longrightarrow S = 0 \longrightarrow Q(S) \stackrel{q}\longrightarrow S$ as the composition of a cofibration followed by a trivial fibration. Then the maps $q \otimes X : Q(S) \otimes X \longrightarrow S \otimes X$ and $X \otimes q : X \otimes Q(S) \longrightarrow X \otimes S$ are weak equivalences for all cofibrant objects $X$. 
\end{itemize}

We show that the Gorenstein-projective model structure is not monoidal, in general. First, we study the module case. Consider the ring $R = \mathbb{Z}_4$. Note that $\mathbb{Z}_4$ is a quasi-Frobenius ring. On the other hand, every module over such a ring is Gorenstein-projective. For if $M$ is a left $R$-module with $R$ quasi-Frobinius, consider a left projective resolution $\cdots \longrightarrow P_1 \stackrel{f_1}\longrightarrow P_0 \stackrel{f_0}\longrightarrow M \longrightarrow 0$ and a right injective resolution $0 \longrightarrow M \stackrel{g^0}\longrightarrow I^0 \stackrel{g^1}\longrightarrow I^1\longrightarrow \cdots$ of $M$. Taking the composition $g^0 \circ f_0$, we have an exact sequence of projective modules $\cdots \longrightarrow P_1 \stackrel{f_1}\longrightarrow P_0 \stackrel{g^0 \circ f_0}\longrightarrow I^0 \stackrel{g^1}\longrightarrow I^1 \longrightarrow \cdots$ such that $M = {\rm Ker}(g^1)$. It is clear that this sequence is ${\rm Hom}_R(-,\mathcal{P}_0) = {\rm Hom}_R(-,\mathcal{I}_0)$-exact. Hence $M$ is Gorenstein-projective. 

Back to the example $R = \mathbb{Z}_4$, we show that the Gorenstein-projective model structure on $\mathbb{Z}_4$-modules is not monoidal.. There exists a left $\mathbb{Z}_4$-module with infinite projective dimension. Recall that the \underline{left global dimension} of a ring $R$ is defined as the value ${\rm sup}\{ {\rm pd}(M) \mbox{ : } M \in \Modl \}$. It is known that if $R$ is a left Noetherian ring, then the left global dimension and the weak dimension of $R$ coincide \cite[Corollary 4.21]{Osborne}. So $\mathbb{Z}_4$ has infinite left global dimension since it is left Noetherian with infinite weak dimension. It follows there exists a Gorenstein-projective $\mathbb{Z}_4$-module $C$ with infinite projective dimension. Consider the trivial cofibration $0 \longrightarrow \mathbb{Z}_4$ and the cofibration $0 \longrightarrow C$. The map $0 \longrightarrow \mathbb{Z}_4 \otimes_{\mathbb{Z}_4} C = (0 \longrightarrow \mathbb{Z}_4) \square (0 \longrightarrow C)$ is not a trivial cofibration, since $\mathbb{Z}_4 \otimes_{\mathbb{Z}_4} C \cong C$ has not finite projective dimension. These arguments also work to show that the Gorenstein-projective model structure on ${\rm {\bf Ch}}(\mbox{}_{\mathbb{Z}_4}{\rm {\bf Mod}})$ is not monoidal. It suffices to consider $D^1(\mathbb{Z}_4)$ and $S^0(C)$, where $D^1(\mathbb{Z}_4) \otimes S^0(C) \cong D^1(C)$, which is Gorenstein-projective with infinite projective dimension. In a similar way, one can show that the Gorenstein-injective model structure is not monoidal on modules or chain complexes over $\mathbb{Z}_4$. It is worth noting in this case that the left global dimension of a ring is also given by ${\rm sup}\{ {\rm id}(M) \mbox{ : } M \in \Modl \}$.


\section{Cotorsion pairs from Gorenstein-projective and Gorenstein-injective dimensions}

The goal in this section is to extend Hovey's Gorenstein-projective and Gorenstein-injective model structures to any Gorenstein-homological dimension. 

Let $\mathcal{F}$ be a class of objects of an Abelian category $\mathcal{C}$ and $X \in {\rm Ob}(\mathcal{C})$. A \underline{left $\mathcal{F}$-resolution} of $X$ is a ${\rm Hom}(\mathcal{F}, -)$-exact (but not necessarily exact) complex $\cdots \longrightarrow F_1 \longrightarrow F_0 \longrightarrow X \longrightarrow 0$ with each $F_i \in \mathcal{F}$. Dually, a ${\rm Hom}(-,\mathcal{F})$-exact complex $0 \longrightarrow X \longrightarrow F^0 \longrightarrow F^1 \longrightarrow \cdots$ with each $F^i \in \mathcal{F}$ is called \underline{right $\mathcal{F}$-resolution} of $X$. 

We say that a map $F \longrightarrow X$ with $F \in \mathcal{F}$ is a \underline{special $\mathcal{F}$-precover} if it is surjective with kernel in $\mathcal{F}^\perp$. The class $\mathcal{F}$ is called  \underline{special pre-covering} if every object of $X$ has a especial $\mathcal{F}$-pre-cover. \underline{Special pre-envelopes} and \underline{special pre-enveloping classes} are defined dually. The following proposition is easy to show. 

\vspace{0.5cm}

\begin{proposition} Let $\mathcal{C}$ be a an abelian category and $\mathcal{F}$ be a class of objects of $\mathcal{C}$. 
\begin{itemize}
\item[ {\bf (1)}] If $\mathcal{F}$ is a special pre-covering class, then every object of $\mathcal{C}$ has an exact left $\mathcal{F}$-resolution. 

\item[ {\bf (2)}] If $\mathcal{F}$ is a special pre-enveloping class, then every object of $\mathcal{C}$ has an exact right $\mathcal{F}$-resolution. 
\end{itemize}
\end{proposition}

\vspace{0.5cm}

\begin{corollary}\label{lapropores} If $\mathcal{C}$ is a Gorenstein category, then every object of $\mathcal{C}$ has an exact left Gorenstein-projective resolution and an exact right Gorenstein-injective resolution. 
\end{corollary}

\vspace{0.5cm}

\begin{definition} {\rm In a Gorenstein category $\mathcal{C}$, it makes sense to consider left $r$-$\mathcal{GP}_0(\mathcal{C})$-objects, where $r$ is a nonnegative integer, since $(\mathcal{GP}_0(\mathcal{C}), \mathcal{W}(\mathcal{C}))$ is a complete cotorsion pair. To simplify, we say that an object $X$ is \underline{Gorenstein-$r$-projective} if it is a left $r$-$\mathcal{GP}_0(\mathcal{C})$-object. Let $\mathcal{GP}_r(\mathcal{C})$ denote the class of Gorenstein-$r$-projective objects. We denote by ${\rm Gpd}(X)$ the (left) Gorenstein-projective dimension of $X$. Notice $\mathcal{GP}_r(\mathcal{C}) = \{ X \in {\rm Ob}(\mathcal{C}) \mbox{ : } {\rm Gpd}(X) \leq r \}$. }
\end{definition}

\begin{proposition}\label{coroproj} Let $\mathcal{C}$ be a Gorenstein category.
\begin{itemize}
\item[ {\bf (1)}] ${\rm Gpd}(X) \leq FDI(\mathcal{C})$ for every object $X \in {\rm Ob}(\mathcal{C})$.

\item[ {\bf (2)}] For every $0 < r \leq FDI(\mathcal{C})$, $\mathcal{GP}_r(\mathcal{C}) \cap \mathcal{W}(\mathcal{C}) = \mathcal{P}_r(\mathcal{C})$. 
\end{itemize} 
\end{proposition}
\begin{proof} To prove the first part, notice that ${\rm id}(W) \leq FDI(\mathcal{C})$, for every $W \in \mathcal{W}(\mathcal{C})$. Then we have ${\rm Ext}^{FDI(\mathcal{C}) + 1}_{\mathcal{C}}(X, W) = 0$. By \cite[Proposition 11.5.7]{EJ}\footnote{Although this result is stated for modules over Gorenstein ring, it carries over to any Gorenstein category.}, we have ${\rm Gpd}(X) \leq FDI(\mathcal{C})$. 

For the second part, it suffices to show that every $X \in \mathcal{GP}_r(\mathcal{C}) \cap \mathcal{W}(\mathcal{C})$ is in $\mathcal{P}_r(\mathcal{C})$. Notice every $G \in \Omega^r(X)$ is in $\mathcal{GP}_0(\mathcal{C})$. Since $\mathcal{W}(\mathcal{C})$ is thick, we also have $G \in \mathcal{W}(\mathcal{C})$. Then $G \in \mathcal{GP}_0(\mathcal{C}) \cap \mathcal{W}(\mathcal{C}) = \mathcal{P}_0(\mathcal{C})$, by \cite[Corollary 8.5]{Hovey}. It follows $X \in \mathcal{P}_r(\mathcal{C})$. 
\end{proof}

\vspace{0.5cm}

\begin{proposition} Let $\mathcal{C}$ be a Gorenstein category. Then $\mathcal{GP}_r(\mathcal{C})$ is the left half of a cotorsion pair $(\mathcal{GP}_r(\mathcal{C}), (\mathcal{GP}_r(\mathcal{C}))^\perp)$, for every $0 < r \leq FDI(\mathcal{C})$. 
\end{proposition}
\begin{proof} It suffices to show the inclusion $\mbox{}^\perp((\mathcal{GP}_r(\mathcal{C}))^\perp) \subseteq \mathcal{GP}_r(\mathcal{C})$. Let $X \in \mbox{}^\perp((\mathcal{GP}_r(\mathcal{C}))^\perp)$ and consider an exact partial left projective resolution of $X$, i.e. an exact sequence of finite length $0 \longrightarrow C \longrightarrow P_{r-1} \longrightarrow \cdots \longrightarrow P_0 \longrightarrow X \longrightarrow 0$ with each $P_i$ projective. By \cite[Proposition 11.5.7]{EJ}, it suffices to show that $C$ is Gorenstein-projective, i.e. that ${\rm Ext}^1_{\mathcal{C}}(C, W) = 0$ for every $W \in \mathcal{W}(\mathcal{C})$. By Dimension Shifting (see \cite[Proposition 4.2]{Osborne}), we have ${\rm Ext}^1_{\mathcal{C}}(C,W) \cong {\rm Ext}^{1+r}_{\mathcal{C}}(X,W)$. On the other hand, ${\rm Ext}^{1+r}_{\mathcal{C}}(X,W) \cong {\rm Ext}^1_{\mathcal{C}}(X, K)$, where $K \in \Omega^{-r}(W)$. It is not hard to see that $K \in (\mathcal{GP}_r(\mathcal{C}))^\perp$. It follows ${\rm Ext}^1_{\mathcal{C}}(X, K) = 0$, and hence ${\rm Ext}^1_{\mathcal{C}}(C,W) = 0$, for every $W \in \mathcal{W}(\mathcal{C})$.  
\end{proof}

\vspace{0.25cm}

It is known that $(\mathcal{P}_r, (\mathcal{P}_r)^\perp)$ is a complete cotorsion pair in the category of left $R$-modules (See \cite[Theorem 7.4.6]{EJ}). In the given reference, the authors prove that $(\mathcal{P}_r, (\mathcal{P}_r)^\perp)$ is a cotorsion pair after showing how to construct $\kappa$-small $r$-projective transfinite extensions for every $r$-projective module, where $\kappa$ is an infinite cardinal satisfying $\kappa > {\rm Card}(R)$. A simpler proof can be given in the context of Gorenstein categories. \\

\begin{corollary} If $\mathcal{C}$ is a Gorenstein category, then $(\mathcal{P}_r(\mathcal{C}), (\mathcal{P}_r(\mathcal{C}))^\perp)$ is a cotorsion pair.
\end{corollary}
\begin{proof}
We only need to show $\mbox{}^\perp((\mathcal{P}_r(\mathcal{C}))^\perp) \subseteq \mathcal{P}_r(\mathcal{C})$, which follows by the implications $\mathcal{P}_r(\mathcal{C}) \subseteq \mathcal{GP}_r(\mathcal{C}) \Longrightarrow \mbox{}^\perp((\mathcal{P}_r(\mathcal{C}))^\perp) \subseteq \mbox{}^\perp((\mathcal{GP}_r(\mathcal{C}))^\perp) = \mathcal{GP}_r(\mathcal{C})$ and $\mathcal{P}_r(\mathcal{C}) \subseteq \mathcal{W}(\mathcal{C}) \Longrightarrow \mbox{}^\perp((\mathcal{P}_r(\mathcal{C}))^\perp) \subseteq \mbox{}^\perp((\mathcal{W}(\mathcal{C}))^\perp) = \mathcal{W}(\mathcal{C})$.
\end{proof}

\newpage

In order to prove that $(\mathcal{GP}_r(\mathcal{C}), (\mathcal{GP}_r(\mathcal{C}))^\perp)$ and $(\mathcal{P}_r(\mathcal{C}), (\mathcal{P}_r(\mathcal{C}))^\perp)$ are compatible, by Proposition \ref{coroproj} it suffices to show $(\mathcal{GP}_r(\mathcal{C}))^\perp = \mathcal{W}(\mathcal{C}) \cap (\mathcal{P}_r(\mathcal{C}))^\perp$. Moreover, we shall see that in the cases $\mathcal{C} = \Modl$ or $\mathcal{C} = \Cadl$, these pairs are also complete. 

\vspace{0.5cm}

\begin{proposition}\label{igualdadesperp} The equality $(\mathcal{GP}_r(\mathcal{C}))^\perp = (\mathcal{P}_r(\mathcal{C}))^\perp \cap \mathcal{W}(\mathcal{C})$ holds in every Gorenstein category $\mathcal{C}$ and for every $0 < r \leq FDI(\mathcal{C})$. 
\end{proposition}
\begin{proof}
The inclusion $(\mathcal{GP}_r(\mathcal{C}))^\perp \subseteq (\mathcal{P}_r(\mathcal{C}))^\perp \cap \mathcal{W}(\mathcal{C})$ follows as in the previous corollary. Now let $Y \in (\mathcal{P}_r(\mathcal{C}))^\perp \cap \mathcal{W}(\mathcal{C})$ and $X \in \mathcal{GP}_r(\mathcal{C})$. Since $(\mathcal{GP}_0(\mathcal{C}), \mathcal{W}(\mathcal{C}))$ is complete, there exists a short exact sequence $0 \longrightarrow X \longrightarrow W \longrightarrow C \longrightarrow 0$ with $C \in \mathcal{GP}_0(\mathcal{C})$ and $W \in \mathcal{W}(\mathcal{C})$. Note $W$ also belongs to $\mathcal{GP}_r(\mathcal{C})$, and thus $W \in \mathcal{P}_r(\mathcal{C})$. On the other hand, we have a long exact sequence of derived functors $\cdots \longrightarrow {\rm Ext}^1_{\mathcal{C}}(W, Y) \longrightarrow {\rm Ext}^1_{\mathcal{C}}(X, Y) \longrightarrow {\rm Ext}^2_{\mathcal{C}}(C, Y) \longrightarrow \cdots$, where ${\rm Ext}^1_{\mathcal{C}}(W, Y) = 0$ since $W \in \mathcal{P}_r(\mathcal{C})$, and ${\rm Ext}^2_{\mathcal{C}}(C, Y) = 0$ by \cite[Proposition 10.2.6]{EJ}. Therefore, $Y \in (\mathcal{GP}_r(\mathcal{C}))^\perp$. 
\end{proof}

\vspace{0.5cm}

\begin{corollary}\label{Gorprojcomplete} If $R$ is an $n$-Gorenstein ring, then $(\mathcal{GP}_r, (\mathcal{GP}_r)^\perp)$ is a complete cotorsion pair in $\Modl$. 
\end{corollary}
\begin{proof} On the one hand, $(\mathcal{GI}_0, \mathcal{W})$ and $(\mathcal{P}_r,(\mathcal{P}_r)^\perp)$ are cogenerated by a sets according to \cite[Theorem 8.3]{Hovey} and \cite[Theorem 7.4.6]{EJ}, respectively. On the other hand, the pairs $(\mathcal{GP}_r,(\mathcal{GP}_r)^\perp)$ and $(\mathcal{P}_r,(\mathcal{P}_r)^\perp)$ are compatible by Propositions \ref{coroproj} and \ref{igualdadesperp}. The completeness of $(\mathcal{GP}_r, (\mathcal{GP}_r)^\perp)$ follows by \cite[Proposition 4.1]{Perez}. 
\end{proof}

\vspace{0.25cm}

Applying Hovey's correspondence, we obtain the following new model structure on $\Modl$. 

\vspace{0.5cm}

\begin{theorem}\label{Gorrprojmodel} If $R$ is an $n$-Gorenstein ring, then for each $0 < r \leq n$ the two compatible and complete cotorsion pairs \[ (\mathcal{GP}_r \cap \mathcal{W}, (\mathcal{P}_r)^\perp) \mbox{ \ and \ } (\mathcal{GP}_r, (\mathcal{P}_r)^\perp \cap \mathcal{W}) \] give rise to a unique Abelian model structure on $\Modl$, called \underline{Gorenstein-$r$-projective} \underline{model structure}, where $\mathcal{GP}_r$, $(\mathcal{P}_r)^\perp$ and $\mathcal{W}$ are the classes of cofibrant, fibrant and trivial objects, respectively. 
\end{theorem}

\vspace{0.5cm}

In \cite[Corollary 4.1]{Perez}, it is proven that $( \widetilde{\mathcal{P}_r}, ( \widetilde{\mathcal{P}_r} )\mbox{}^\perp )$ is a cotorsion pair cogenerated by the set $(\widetilde{\mathcal{P}_r})^{\leq \kappa} := \{ X \in \widetilde{\mathcal{P}_r} \mbox{ : } {\rm Card}(X) \leq \kappa \}$, with $\kappa$ an infinite cardinal satisfying $\kappa > {\rm Card}(R)$. If we denote $\widehat{\mathcal{GP}_r} = \mathcal{GP}_r(\Cadl)$, the following result can be proven as in $\Modl$. 

\newpage

\begin{corollary} If $R$ is an $n$-Gorenstein ring, then $( \widehat{\mathcal{GP}_r}, ( \widehat{\mathcal{GP}_r} )\mbox{}^\perp )$ is a complete cotorsion pair in $\Cadl$. 
\end{corollary}

\vspace{0.25cm}

By the completeness of $( \widetilde{\mathcal{P}_r}, ( \widetilde{\mathcal{P}_r} )\mbox{}^\perp )$ and $( \widehat{\mathcal{GP}_r}, ( \widehat{\mathcal{GP}_r} )\mbox{}^\perp )$, along with Propositions \ref{coroproj} and \ref{igualdadesperp}, we obtain the analogue of Theorem \ref{Gorrprojmodel} for $\Cadl$. 

\vspace{0.5cm}

\begin{theorem} If $R$ is an $n$-Gorenstein ring, then for each $0 < r \leq n$ the two compatible and complete cotorsion pairs \[ (\widehat{\mathcal{GP}_r} \cap \widetilde{\mathcal{W}}, (\widetilde{\mathcal{P}_r})^\perp) \mbox{ \ and \ } (\widehat{\mathcal{GP}_r}, (\widetilde{\mathcal{P}_r})^\perp \cap \widetilde{\mathcal{W}}) \] give rise to a unique Abelian model structure on $\Cadl$, where $\widehat{\mathcal{GP}_r}$, $(\widetilde{\mathcal{P}_r})^\perp$ and $\widetilde{\mathcal{W}}$ are the classes of cofibrant, fibrant and trivial objects, respectively. 
\end{theorem}

\vspace{0.25cm}

It is time to construct a model structure from the notion of Gorenstein-injective dimension. The advantage in this case is that our results can be presented for every Gorenstein category, not only modules and complexes over Gorenstein rings. As we noted for the class $\mathcal{GP}_0(\mathcal{C})$, it makes sense to consider right $r$-$\mathcal{GI}_0(\mathcal{C})$-objects in any Gorenstein category $\mathcal{C}$. 

\vspace{0.5cm}

\begin{definition} We say that an object $X$ of a Gorenstein category $\mathcal{C}$ is \underline{Gorenstein-$r$-injective} if it is a right $r$-$\mathcal{GI}_0(\mathcal{C})$-object. We denote by $\mathcal{GI}_r(\mathcal{C})$ the class of Gorenstein-$r$-injective objects of $\mathcal{C}$. Let ${\rm Gid}(X)$ denote the (right) Gorenstein-injective dimension of $X$. Notice the equality $\mathcal{GI}_r(\mathcal{C}) = \{ X \in {\rm Ob}(\mathcal{C}) \mbox{ : } {\rm Gid}(X) \leq r \}$. 
\end{definition}

\vspace{0.25cm}

As we did in the Gorenstein-projective case, we have the following results: 

\vspace{0.5cm}

\begin{proposition}\label{gorinjpairs} Let $\mathcal{C}$ be a Gorenstein category. The following conditions hold for every $0 < r \leq FDP(\mathcal{C})$. 
\begin{itemize}
\item[{\bf (1)}] ${\rm Gid}(X) \leq FDP(\mathcal{C})$, for every object $X$ of $\mathcal{C}$.

\item[{\bf (2)}] $\mathcal{GI}_r(\mathcal{C}) \cap \mathcal{W}(\mathcal{C}) = \mathcal{I}_r(\mathcal{C})$.

\item[{\bf (3)}] $(\mbox{}^\perp(\mathcal{GI}_r(\mathcal{C})), \mathcal{GI}_r(\mathcal{C}))$ is a cotorsion pair.

\item[{\bf (4)}] $\mbox{}^\perp(\mathcal{GI}_r(\mathcal{C})) = \mbox{}^\perp(\mathcal{I}_r(\mathcal{C})) \cap \mathcal{W}(\mathcal{C})$. 
\end{itemize} 
\end{proposition}

\vspace{0.5cm}

\begin{theorem} Let $\mathcal{C}$ be a Gorenstein category. 
\begin{itemize}
\item[ {\bf (1)}] $(\mbox{}^\perp(\mathcal{GI}_r(\mathcal{C})), \mathcal{GI}_r(\mathcal{C}))$ is a complete cotorsion pair. 

\item[ {\bf (2)}] If in addition $\mathcal{C}$ is locally Noetherian, then $(\mbox{}^\perp(\mathcal{GI}_r(\mathcal{C})), \mathcal{GI}_r(\mathcal{C}))$ is cogenerated by the set $\mathcal{S}(r)$ of all $S \in \Omega^i(J)$ with $i \geq r$ and $J$ running over the set of all indecomposable injective objects. 
\end{itemize}
\end{theorem}
\begin{proof} Part {\bf (1)} follows by Proposition \ref{gorinjpairs} and \cite[Proposition 4.1]{Perez}, since the pairs $(\mbox{}^\perp(\mathcal{W}(\mathcal{C})), \mathcal{W}(\mathcal{C}))$, $(\mathcal{W}(\mathcal{C}), (\mathcal{W}(\mathcal{C}))^\perp)$, and $(\mbox{}^\perp(\mathcal{I}_r(\mathcal{C})), \mathcal{I}_r(\mathcal{C}))$ are complete. 

Now suppose $\mathcal{C}$ is also locally Noetherian. We shall see that $\mathcal{GI}_r(\mathcal{C}) = (\mathcal{S}(r))^\perp$. First, we check $\mathcal{S}(r) \subseteq \mbox{}^\perp(\mathcal{GI}_r(\mathcal{C}))$. Let $S \in \mathcal{S}(r)$ and consider $Y \in \mathcal{GI}_r(\mathcal{C})$. Then $S \in \Omega^i(J)$, for some $i \geq r$ and some indecomposable injective object $J$. We have ${\rm Ext}^1_{\mathcal{C}}(S, Y) \cong {\rm Ext}^{i+1}_{\mathcal{C}}(J, Y) = 0$, since $J \in \mathcal{W}(\mathcal{C})$, $Y \in \mathcal{GI}_r(\mathcal{C})$ and $i +1 > r$. Now to prove $\mathcal{GI}_r(\mathcal{C}) = (\mathcal{S}(r))^\perp$, it suffices to show $(\mathcal{S}(r))^\perp \subseteq \mathcal{GI}_r(\mathcal{C})$. Let $Y \in (\mathcal{S}(r))^\perp$ and consider an exact partial right injective resolution of $Y$, say $0 \longrightarrow Y \longrightarrow I^0 \longrightarrow \cdots \longrightarrow I^{r-1} \longrightarrow D \longrightarrow 0$. By \cite[Proposition 11.2.5]{EJ}, we only need to prove $D$ is Gorenstein-injective, i.e. that ${\rm Ext}^1_{\mathcal{C}}(S, D) = 0$ for every $S \in \mathcal{S}$, where $\mathcal{S}$ is the set considered in Proposition \ref{hoveygorinjcat}. By Dimension Shifting, we have ${\rm Ext}^1_{\mathcal{C}}(S, D) \cong {\rm Ext}^{1+r}_{\mathcal{C}}(S, Y) \cong {\rm Ext}^1_{\mathcal{C}}(S', Y)$, where $S' \in \Omega^{r}(S)$. Since $S \in \Omega^i(J)$, for some $i \geq 0$ and some indecomposable injective object $J$, we have $S' \in \Omega^{r}(\Omega^{i}(J)) = \Omega^{r+i}(J)$ with $r+i \geq r$. It follows $S' \in \mathcal{S}(r)$ and ${\rm Ext}^{1}_{\mathcal{C}}(S',Y) = 0$. Therefore, ${\rm Ext}^1_{\mathcal{C}}(S,D) = 0$ for every $S \in \mathcal{S}$. 
\end{proof}

\vspace{0.25cm}

From the results above and Hovey's correspondence, we deduce the existence of the following model structure. 

\vspace{0.5cm}

\begin{theorem} Let $\mathcal{C}$ be a Gorenstein category. If $0 < r \leq FDP(\mathcal{C})$, then the two compatible and complete cotorsion pairs \[ (\mbox{}^\perp(\mathcal{I}_r(\mathcal{C})) \cap \mathcal{W}(\mathcal{C}), \mathcal{GI}_r(\mathcal{C})) \mbox{ \ and \ } (\mbox{}^\perp(\mathcal{I}_r(\mathcal{C})), \mathcal{GI}_r(\mathcal{C}) \cap \mathcal{W}(\mathcal{C})) \] give rise to a unique Abelian model structure on $\mathcal{C}$, called the \underline{Gorenstein-$r$-injective model structure}, such that $\mbox{}^\perp(\mathcal{I}_r(\mathcal{C}))$, $\mathcal{GI}_r(\mathcal{C})$ and $\mathcal{W}(\mathcal{C})$ are the classes of cofibrant, fibrant and trivial objects, respectively. 
\end{theorem} 

\vspace{0.25cm}

Proposition \ref{Gorcomplexes} is also valid for every Gorenstein homological dimension. 

\vspace{0.5cm}

\begin{proposition} Let $X$ be a chain complex over a Gorenstein category $\mathcal{C}$. 
\begin{itemize}
\item[ {\bf (1)}] For every $0 < r \leq FDI(\mathcal{C})$, $X$ is Gorenstein-$r$-projective if, and only if, $X_m$ is a Gorenstein-$r$-projective object for every $m \in \mathbb{Z}$.
\item[ {\bf (2)}] For every $0 < r \leq FDP(\mathcal{C})$, $X$ is Gorenstein-$r$-injective if, and only if, $X_m$ is a Gorenstein-$r$-injective object for every $m \in \mathbb{Z}$. 
\end{itemize} 
\end{proposition} 
\begin{proof} 
We only prove {\bf (1)}. Let $X$ be a Gorenstein $r$-projective chain complex. Then there exists an exact sequence $0 \longrightarrow C_r \longrightarrow C_{r-1} \longrightarrow \cdots  \longrightarrow C_0 \longrightarrow X \longrightarrow 0$ in $\Complexes$ such that each $C_i$ is a Gorenstein-projective complex. Note that for each $m \in \mathbb{Z}$ we have an exact sequence in $0 \longrightarrow (C_r)_m \longrightarrow (C_{r-1})_m \longrightarrow \cdots \longrightarrow (C_0)_m \longrightarrow X_m \longrightarrow 0$ in $\mathcal{C}$, which turns out to be an exact left Gorenstein-projective resolution of $X_m$ of length $r$, by Proposition \ref{Gorcomplexes}. 

Now suppose each $X_m$ is a Gorenstein-$r$-projective object. Consider an exact partial left projective resolution $0 \longrightarrow C \longrightarrow P_{r-1} \longrightarrow \cdots  \longrightarrow P_0 \longrightarrow X \longrightarrow 0$. By \cite[Proposition 11.5.7]{EJ}, it suffices to show that $C$ is a Gorenstein-projective chain complex. Notice that for every $m \in \mathbb{Z}$ we have an exact sequence $0 \longrightarrow (C)_m \longrightarrow (P_{r-1})_m \longrightarrow \cdots \longrightarrow (P_0)_m \longrightarrow X_m \longrightarrow 0$  in $\mathcal{C}$ with each $(P_i)_m$ projective. Since every $X_m \in \mathcal{GP}_r(\mathcal{C})$, we have $C_m$ is Gorenstein-projective in $\mathcal{C}$ by \cite[Proposition 11.5.7]{EJ}. Hence $C$ is a Gorenstein-projective complex by Proposition \ref{Gorcomplexes}. 
\end{proof}


\section{Gorenstein-flat dimension and model structures on mo-dules and chain complexes}

In addition to the notions of Gorenstein-projective and Gorenstein-injective modules, there is also the notion of Gorenstein-flat modules. In \cite{HoveyGillespie}, M. Hovey and J. Gillespie construct a unique Abelian model structure on $\Modl$ (with $R$ a Gorenstein ring) where the cofibrant objects are the Gorenstein-flat modules, the fibrant objects are the cotorsion modules, and the class $\mathcal{W}$ is again the class of trivial objects. We call this structure the \underline{Gorenstein-flat model structure}. We shall generalize the construction of this structure to the category of complexes, using the notion of bar-cotorsion pairs. 

We begin this section studying the concept of Gorenstein-flat dimension, in order to construct related model structures on both $\Modl$ and $\Cadl$. 

\vspace{0.5cm}

\begin{definition} {\rm A left $R$-module $E \in {\rm Ob}(\Modl)$ is said to be \underline{Gorenstein-flat} if there exists an exact sequence of flat modules $\cdots \longrightarrow F_1 \longrightarrow F_0 \longrightarrow F^0 \longrightarrow F^1 \longrightarrow \cdots$ such that $E = {\rm Ker}(F^0 \longrightarrow F^1)$ and such that the sequence $\cdots \longrightarrow I \otimes_R F_1 \longrightarrow I \otimes_R F_0 \longrightarrow I \otimes_R F^0 \longrightarrow I \otimes_R F^{1} \longrightarrow \cdots$ is exact, for every injective module $I$. }
\end{definition}

\newpage

Let $\mathcal{GF}_0$ denote the class of Gorenstein-flat left $R$-modules. We class $\mathcal{GC} = (\mathcal{GF}_0)\mbox{}^\perp$ is called the class of \underline{Gorenstein-cotorsion modules}. It is known that $(\mathcal{GF}_0, \mathcal{GC})$ is a complete cotorsion pair, as mentioned in \cite{HoveyGillespie}. For our purposes, we shall provide a cogenerating set for this pair, which is related to the following restriction of the notion of pure sub-modules in the context of Gorenstein homological algebra.  

\vspace{0.5cm}

\begin{definition} {\rm A sub-module $N$ of a left $R$-module $M$ is said to be \underline{$\mathcal{W}$-pure} if for every right $R$-module $W \in \mathcal{W}$, the sequence $0 \longrightarrow W \otimes_R N \longrightarrow W \otimes_R M$ is exact. The short exact sequence $0 \longrightarrow N \longrightarrow M \longrightarrow M/N \longrightarrow 0$ is called a \underline{$\mathcal{W}$-pure exact sequence}. }
\end{definition}

\vspace{0.5cm}

\begin{proposition}\label{subcos} Let $R$ be a Gorenstein ring. If $S$ is a $\mathcal{W}$-pure sub-module of a Gorenstein-flat left $R$-module $E$, then $S$ and $E/S$ are also Gorenstein-flat. 
\end{proposition}
\begin{proof} Consider the short exact sequence $0 \longrightarrow S \longrightarrow E \longrightarrow E / S \longrightarrow 0$. Then, given a right $R$-module $W \in \mathcal{W}$, we have the long exact sequence \[ \cdots \longrightarrow {\rm Tor}^R_1(W, E / S) \longrightarrow W \otimes_R S \longrightarrow W \otimes_R E \longrightarrow W \otimes_R E / S \longrightarrow 0. \] Since  $S$ is a $\mathcal{W}$-pure sub-module of $E$, we have that the map $W \otimes_R S \longrightarrow W \otimes_R E$ is injective for every $W \in \mathcal{W}$. So ${\rm Tor}^R_1(W, E / S) = 0$ for every $W \in \mathcal{W}$. Hence $E / S$ is a Gorenstein-flat module by \cite[Theorem 10.3.8]{EJ}. Using the same reference and the previous long exact sequence, we can also conclude that $S$ is Gorenstein-flat.
\end{proof}

\vspace{0.5cm}

\begin{proposition} Given an infinite cardinal $\kappa > {\rm Card}(R)$, a Gorenstein-flat module $E$ and an element $x \in E$, there exists a $\mathcal{W}$-pure sub-module $S \subseteq E$ such that $x \in S$ and ${\rm Card}(S) \leq \kappa$.
\end{proposition}
\begin{proof} The proof is the same as the proof given in \cite[Proposition 7.4.3]{EJ}.
\end{proof}

\vspace{0.25cm}

By the previous two propositions, given $E \in \mathcal{GF}_0$, we can construct a continuos chain $\{ S_\alpha \}_{\alpha < \lambda}$ of sub-modules of $E$, for some ordinal $\lambda$, such that $S_1 \in (\mathcal{GF}_0)^{\leq \kappa}$, $S_{\alpha + 1} / S_\alpha \in (\mathcal{GF}_0)^{\leq \kappa}$, and $S_\beta = \bigcup_{\alpha < \beta} S_\alpha$, for every limit ordinal $\beta < \alpha$, where $(\mathcal{GF}_0)^{\leq \kappa} := \{ S \in \mathcal{GF}_0 \mbox{ : } {\rm Card}(S) \leq \kappa \}$. In other words, every Gorenstein-flat left $R$-module $E$ has a $(\mathcal{GF}_0)^{\leq \kappa}$-filtration. By \cite[Proposition 3.2]{Perez}, $(\mathcal{GF}_0, \mathcal{GC})$ is a cotorsion pair cogenerated by $(\mathcal{GF}_0)^{\leq \kappa}$. As a consequence of the Eklof and Trlifaj Theorem,  $(\mathcal{GF}_0, \mathcal{GC})$ is a complete.

In order to present the Gorenstein-flat model structure on the category of chain complexes, we shall introduce the notion of bar-cotorsion pairs. Previously, we have defined the bar-extension functors $\overline{{\rm Ext}}^i(-,-)$ for chain complexes. Bar-cotorsion pairs are defined by two classes of chain complexes orthogonal to each with respect to $\overline{{\rm Ext}}^1(-,-)$, instead of ${\rm Ext}^1(-,-)$. 

\vspace{0.5cm}

\begin{definition} {\rm Let $\mathcal{A}$ and $\mathcal{B}$ be two classes of chain complexes over $\Modl$, we shall say that $\mathcal{A}$ and $\mathcal{B}$ form a \underline{bar-cotorsion pair} $(\mathcal{A}\mbox{ }|\mbox{ }\mathcal{B})$ if:
\begin{itemize}
\item[ {\bf (1)}] $\mathcal{A} = \mbox{}^{\overline{\perp}}\mathcal{B} := \{ A \in {\rm Ob}(\Cadl) \mbox{ : } \overline{{\rm Ext}}^1(A, B) = 0, \mbox{ for every }B \in \mathcal{B} \}$,

\item[ {\bf (2)}] $\mathcal{B} = \mathcal{A}^{\overline{\perp}} := \{ B \in {\rm Ob}(\Cadl) \mbox{ : } \overline{{\rm Ext}}^1(A, B) = 0, \mbox{ for every }A \in \mathcal{A} \}$. 
\end{itemize} }
\end{definition}

\vspace{0.5cm}

The notion of suspensions of chain complexes is useful to establish some properties of bar-cotorsion pairs. Given a chain complex $X$ and an integer $k \in \mathbb{Z}$, the \underline{$k$th suspension} of $X$ is the complex $\Sigma^k(X)$ given by $(\Sigma^k(X))_n := X_{n-k}$, where the boundary maps are defined by $\partial^{\Sigma^k(X)}_n := (-1)^k \partial^X_{n-k}$. A class of complexes $\mathcal{D}$ is said to be \underline{closed under suspensions} if $\Sigma^k(D) \in \mathcal{D}$, for every complex $D \in \mathcal{D}$. 

The importance of suspension functors lies in the fact that each $\overline{{\rm Ext}}^i(-,-)$ can be expressed in terms of suspensions and the standard ${\rm Ext}^i(-,-)$. Specifically, for every pair $X, Y \in {\rm Ob}(\Cadl)$ and every $i \geq 0$, $\overline{{\rm Ext}}^i(X, Y)$ is a chain complex of the form \[ \cdots \longrightarrow {\rm Ext}^i(X, \Sigma^{k+1}(Y)) \longrightarrow {\rm Ext}^i(X, \Sigma^k(Y)) \longrightarrow {\rm Ext}^i(X, \Sigma^{k-1}(Y)) \longrightarrow \cdots \] This fact is mentioned in \cite[Page 87]{GR} and it is not hard to prove. There is also an isomorphism of groups ${\rm Ext}^i(\Sigma^{-k}(X), Y) \cong {\rm Ext}^i(X, \Sigma^k(Y))$, for every $k \in \mathbb{Z}$, as mentioned in \cite[Page 2]{Rada}. With these facts in mind, we can prove the following lemma. 

\vspace{0.5cm}

\begin{lemma}\label{cerradura} Let $(\mathcal{A} \mbox{ }|\mbox{ }\mathcal{B})$ be a bar-cotorsion pair. Then $\mathcal{A}$ is closed under suspensions if, and only if, $\mathcal{B}$ is. The same result holds if $\mathcal{A}$ and $\mathcal{B}$ form a cotorsion pair $(\mathcal{A}, \mathcal{B})$. 
\end{lemma}
\begin{proof} Suppose $(\mathcal{A} \mbox{ }|\mbox{ }\mathcal{B})$ is a bar-cotorsion pair and that $\mathcal{A}$ is closed under suspensions. Let $B \in \mathcal{B}$, $A \in \mathcal{A}$ and $m \in \mathbb{Z}$. Note that the complexes
\[ \begin{tikzpicture}
\matrix (m) [matrix of math nodes, row sep=1em, column sep=1.5em]
{ \cdots & {\rm Ext}^1(A, \Sigma^{k+1}(\Sigma^m(B))) & {\rm Ext}^1(A, \Sigma^k(\Sigma^m(B))) & {\rm Ext}^1(A, \Sigma^{k-1}(\Sigma^m(B))) & \cdots \\ };
\path[->]
(m-1-1) edge (m-1-2) (m-1-2) edge (m-1-3) (m-1-3) edge (m-1-4) (m-1-4) edge (m-1-5);
\end{tikzpicture} \]
and \newpage
\[ \begin{tikzpicture}
\matrix (m) [matrix of math nodes, row sep=1em, column sep=1.5em]
{ \cdots & {\rm Ext}^1(\Sigma^{-m}(A), \Sigma^{k + 1}(B)) & {\rm Ext}^1(\Sigma^{-m}(A), \Sigma^{k}(B)) & {\rm Ext}^1(\Sigma^{-m}(A), \Sigma^{k-1}(B)) & \cdots \\ };
\path[->]
(m-1-1) edge (m-1-2) (m-1-2) edge (m-1-3) (m-1-3) edge (m-1-4) (m-1-4) edge (m-1-5);
\end{tikzpicture} \]
are isomorphic. By the comments above, we have $\overline{{\rm Ext}}^1(A, \Sigma^m(B)) \cong \overline{{\rm Ext}}^1(\Sigma^{-m}(A), B)$. On the other hand, $\overline{{\rm Ext}}^1(\Sigma^{-m}(A), B)= 0$ for every $A \in \mathcal{A}$, since $B \in \mathcal{B}$ and $\mathcal{A}$ is closed under suspensions. Hence the result follows. The converse can be proven similarly.  
\end{proof}

\vspace{0.25cm}

\begin{theorem}\label{teobarra} Let $\mathcal{A}$ and $\mathcal{B}$ be two classes in $\Cadl$ such that $\mathcal{A}$ is closed under suspensions. Then $(\mathcal{A} \mbox{ }|\mbox{ }\mathcal{B})$ is a bar-cotorsion pair if, and only if, $(\mathcal{A, B})$ is a cotorsion pair.
\end{theorem}
\begin{proof} Suppose $(\mathcal{A} \mbox{ }|\mbox{ }\mathcal{B})$ is a bar-cotorsion pair. First, we show $\mathcal{A} = \mbox{}^\perp\mathcal{B}$. Let $A \in \mathcal{A}$ and $B \in \mathcal{B}$. Notice that $\overline{{\rm Ext}}^1(A, B) = 0$ implies ${\rm Ext}^1(A, B) = 0$, and so $A \in \mbox{}^\perp\mathcal{B}$. Now if $A \in \mbox{}^\perp\mathcal{B}$, we have ${\rm Ext}^1(A, B) = 0$ for every $B \in \mathcal{B}$. Since $\mathcal{B}$ is closed under suspensions by Lemma \ref{cerradura}, we have $\Sigma^k(B) \in \mathcal{B}$ for every $k \in \mathbb{Z}$. Then ${\rm Ext}^1(A, \Sigma^k(B)) = 0$ for every $k \in \mathbb{Z}$. It follows $\overline{{\rm Ext}}^1(A, B) = 0$ for every $B \in \mathcal{B}$ and so $A \in \mbox{}^{\overline{\perp}}\mathcal{B} = \mathcal{A}$. The equality $\mathcal{B} = \mathcal{A}^\perp$ follows similarly and using the isomorphism ${\rm Ext}^1(A, \Sigma^k(B)) \cong {\rm Ext}^1(\Sigma^{-k}(A), B)$.  
\end{proof}

\vspace{0.5cm}

\begin{definition} {\rm A bar-cotorsion pair $(\mathcal{A} \mbox{ }|\mbox{ } \mathcal{B})$ is said to be \underline{bar-cogenerated} by a set $\mathcal{S} \subseteq \mathcal{A}$ if $\mathcal{B} = \mathcal{S}^{\overline{\perp}}$. }
\end{definition}

\vspace{0.5cm}

\begin{theorem} Let $\mathcal{A}$ and $\mathcal{B}$ be two classes in $\Cadl$, and $\mathcal{S} \subseteq \mathcal{A}$ be a set closed under suspensions. If $(\mathcal{A, B})$ is a cotorsion pair cogenerated by $\mathcal{S}$, then $\mathcal{A}$ and $\mathcal{B}$ form a bar-cotorsion pair $(\mathcal{A} \mbox{ }|\mbox{ }\mathcal{B})$ which is bar-cogenerated by $\mathcal{S}$. The converse is also true. 
\end{theorem}
\vspace{-0.25cm}
\begin{proof} Suppose $(\mathcal{A, B})$ is cogenerated by $\mathcal{S}$. Then $\mathcal{B} = \mathcal{S}^\perp$, which is closed under suspensions by Lemma \ref{cerradura} (since $\mathcal{S}$ is closed under suspensions). By the same lemma, $\mathcal{A}$ is also closed under suspensions, and hence by Theorem \ref{teobarra} we have $(\mathcal{A} \mbox{ }|\mbox{ }\mathcal{B})$ is a bar-cotorsion pair. It is only left to show that $\mathcal{B} = \mathcal{S}^{\overline{\perp}}$. Let $B \in \mathcal{B}$. Then ${\rm Ext}^1(S, B) = 0$ for every $S \in \mathcal{S}$. Since $\mathcal{S}$ is closed under suspensions, we have $\Sigma^{-k}(S) \in \mathcal{S}$ for every $k \in \mathbb{Z}$ and every $S \in \mathcal{S}$. Then ${\rm Ext}^1(S, \Sigma^k(B)) \cong {\rm Ext}^1(\Sigma^{-k}(S), B) = 0$, for every $k \in \mathbb{Z}$ and every $S \in \mathcal{S}$. It follows $\overline{{\rm Ext}}^1(S, B) = 0$ for every $S \in \mathcal{S}$, i.e. $B \in \mathcal{S}^{\overline{\perp}}$. Now suppose $B \in \mathcal{S}^{\overline{\perp}}$. Then $\overline{{\rm Ext}}^1(S, B) = 0$ for every $S \in \mathcal{S}$. It follows ${\rm Ext}^1(S, B) = 0$ for every $S \in \mathcal{S}$ and so $B \in \mathcal{S}^\perp = \mathcal{B}$. The converse follows similarly.  
\end{proof}

\vspace{0.25cm}

\begin{lemma}[Eklof's Lemma for bar-extensions] Let $X$ be a chain complex in $\Cadl$ and $(S_\alpha)_{\alpha < \lambda}$ be a continuous chain of sub-complexes of $X$. If $\overline{{\rm Ext}}^1(S_{\alpha + 1} / S_\alpha, Y) = 0$ for every $\alpha + 1 < \lambda$, then $\overline{{\rm Ext}}^1(X, Y) = 0$.
\end{lemma}
\begin{proof} This can be proved easily by expressing $\overline{{\rm Ext}}^1(-,-)$ in terms of ${\rm Ext}^1(-,-)$ and $\Sigma$, and using the original Eklof Lemma for ${\rm Ext}^1(-,-)$\footnote{A proof is given in \cite[Lemma 3.1.2]{Gobel} for the category $\Modl$, but the same argument carries over to $\Cadl$.}. 
\end{proof}

\vspace{0.25cm}

The following result follows. 

\vspace{0.5cm}

\begin{proposition}[{\cite[Proposition 1.1]{Perez}} for bar-cotorsion pairs] \label{cogeneracion} Let $(\mathcal{A} \mbox{ }|\mbox{ } \mathcal{B})$ be a bar-cotorsion pair in $\Cadl$, and $\mathcal{S} \subseteq \mathcal{A}$ be a set. If every $X \in \mathcal{A}$ has an $\mathcal{S}$-filtration, then $(\mathcal{A} \mbox{ }|\mbox{ } \mathcal{B})$ is bar-cogenerated by $\mathcal{S}$. 
\end{proposition}

\vspace{0.25cm}

Now we are ready to construct the Gorenstein-flat model structure on $\Cadl$. 

\vspace{0.5cm}

\begin{definition} {\rm A chain complex $E \in {\rm Ob}(\Cadl)$ is said to be \underline{Gorenstein-flat} if there exists an exact sequence of flat complexes $\cdots \longrightarrow F_1 \longrightarrow F_0 \longrightarrow F^0 \longrightarrow F^1 \longrightarrow \cdots$ such that $E = {\rm Ker}(F^0 \longrightarrow F^1)$ and such that the sequence $\cdots \longrightarrow I \overline{\otimes} F_1 \longrightarrow I \overline{\otimes} F_0 \longrightarrow I \overline{\otimes} F^0 \longrightarrow I \overline{\otimes} F^{1} \longrightarrow \cdots$ is exact, for every injective complex $I$. We shall denote the class of all Gorenstein-flat complexes by $\widehat{\mathcal{GF}_0}$. }
\end{definition}

\vspace{0.25cm}

Assume that $R$ is a commutative $n$-Gorenstein ring. First, recall from \cite[Definition 4.1.2]{GR} that a chain complex $F$ is \underline{flat} if $F$ is exact and $Z_m(F)$ is a flat module, for every $m \in \mathbb{Z}$. In \cite[Proposition 5.1.2]{GR}, the author proved that $F$ is flat if, and only if, $-\overline{\otimes} F$ is an exact functor. 

\vspace{0.5cm}

\begin{definition} {\rm A complex $Y \in \Cadl$ is \underline{Gorenstein cotorsion} if $\overline{\rm Ext}^1(X, Y) = 0$ for every Gorenstein-flat complex $X \in \widehat{\mathcal{GF}_0}$. We shall denote the class of all Gorenstein cotorsion complexes by $\widehat{\mathcal{GC}}$. }
\end{definition}

\vspace{0.5cm}

\begin{proposition} Let $R$ be a commutative Gorenstein ring, then $\widehat{\mathcal{GF}_0}$ and $\widehat{\mathcal{GC}}$ form a bar-cotorsion pair $(\widehat{\mathcal{GF}_0} \mbox{ }|\mbox{ } \widehat{\mathcal{GC}})$. 
\end{proposition}
\begin{proof} We only have to show that $\mbox{}^{\overline{\perp}}\widehat{\mathcal{GC}} \subseteq \widehat{\mathcal{GF}_0}$. Let $X \in \mbox{}^{\overline{\perp}}\widehat{\mathcal{GC}}$. By \cite[Theorem 5.4.3]{GR}, we only need to show $\overline{\rm Tor}_1(W, X) = 0$, for every $W \in \widetilde{\mathcal{W}}$. So let $W \in \widetilde{\mathcal{W}}$. By Proposition \ref{isomorfos}, we have $\overline{\rm Tor}_1(W, X)^+ \cong \overline{\rm Tor}_1(X, W)^+ \cong \overline{\rm Ext}^1(X, W^+)$. Now let $F$ be a Gorenstein-flat complex. Then $\overline{\rm Ext}^1(F, W^+) \cong \overline{\rm Tor}_1(W, F)^+ = 0$. We have $W^+ \in \widehat{\mathcal{GC}}$ and so $\overline{\rm Ext}^1(X, W^+) = 0$. It follows $\overline{\rm Tor}_1(W, X) = 0$ and $X \in \widehat{\mathcal{GF}_0}$. 
\end{proof}

\newpage

Using \cite[Theorem 5.4.3]{GR}, we can show that $\widehat{\mathcal{GF}_0}$ is closed under suspensions. So by the previous proposition, Lemma \ref{cerradura} and Theorem \ref{teobarra}, $(\widehat{\mathcal{GF}_0}, \widehat{\mathcal{GC}})$ is a cotorsion pair. 

Let ${\rm dw}\widetilde{\mathcal{F}_0}$ (resp. ${\rm ex}\widetilde{\mathcal{F}_0}$) denote the class of (resp. exact) chain complexes $X \in \Cadl$ such that $X_m \in \mathcal{F}_0$ for every $m \in \mathbb{Z}$. In \cite[Proposition 4.1]{Aldrich} it is proven that for every chain complex $X \in {\rm dw}\widetilde{\mathcal{F}_0}$ and every $x \in X$ (i.e. $x \in X_m$ for some $m \in \mathbb{Z}$), there exists a sub-complex $S \subseteq X$ in $({\rm dw}\widetilde{\mathcal{F}_0})^{\leq \kappa}$ such that $x \in S$ and $X / S \in {\rm dw}\widetilde{\mathcal{F}_0}$. By $S \in ({\rm dw}\widetilde{\mathcal{F}_0})^{\leq \kappa}$ we mean $S \in {\rm dw}\widetilde{\mathcal{F}_0}$ and ${\rm Card}(S) := \sum_{m \in \mathbb{Z}} {\rm Card}(S_m) \leq \kappa$, where $\kappa$ is an infinite cardinal satisfying $\kappa > {\rm Card}(R)$. The same argument can be applied to the class $\widehat{\mathcal{GF}_0}$, using \cite[Theorem 5.4.3]{GR}, in order to get the following result. 

\vspace{0.5cm}

\begin{proposition} Let $\kappa$ be an infinite cardinal satisfying $\kappa > {\rm Card}(R)$, where $R$ is an Gorenstein ring. Let $X$ be a Gorenstein-flat complex and $x \in X$. Then there exists a Gorenstein-flat sub-complex $S \subseteq X$ with ${\rm Card}(S) \leq \kappa$, such that $x \in S$ and $X / S$ is also Gorenstein-flat. 
\end{proposition}

\vspace{0.25cm}

Using this result, every Gorenstein-flat complex has a $(\widehat{\mathcal{GF}_0})^{\leq \kappa}$-filtration, and hence $(\widehat{\mathcal{GF}_0}, \widehat{\mathcal{GC}})$ is a cotorsion pair cogenerated by $(\widehat{\mathcal{GF}_0})^{\leq \kappa}$. Let $\widetilde{\mathcal{F}_0}$ denote the class of flat chain complexes. In \cite[Proposition 4.9]{Gil}, the author proves that $(\widetilde{\mathcal{F}_0}, (\widetilde{\mathcal{F}_0})\mbox{}^\perp)$ is a cotorsion pair cogenerated by $(\widetilde{\mathcal{F}_0})^{\leq \kappa} := \{ F \in \widetilde{\mathcal{F}_0} \mbox{ : } {\rm Card}(F) \leq \kappa \}$. We show that $(\widehat{\mathcal{GF}_0}, \widehat{\mathcal{GC}})$ and $(\widetilde{\mathcal{F}_0}, (\widetilde{\mathcal{F}_0})^\perp)$ are compatible. 

\vspace{0.5cm}

\begin{proposition} $\widetilde{\mathcal{F}_0} = \widehat{\mathcal{GF}_0} \cap \widetilde{\mathcal{W}}$. 
\end{proposition}
\begin{proof} The inclusion $\widetilde{\mathcal{F}_0} \subseteq \widehat{\mathcal{GF}_0} \cap \widetilde{\mathcal{W}}$ is clear by \cite[Theorem 5.4.3]{GR} and \cite[Theorem 9.1.10]{EJ}. Now let $X \in \widehat{\mathcal{GF}_0} \cap \widetilde{\mathcal{W}}$. Then $X^+ \in \widehat{\mathcal{GI}_0}$ by \cite[Theorem 5.4.3]{GR}, where $\widehat{\mathcal{GI}_0}$ denotes the class of Gorenstein-injective complexes. On the other hand, ${\rm fd}(X) = k < \infty$, and so there exists an exact sequence $0 \longrightarrow F_k \longrightarrow F_{k-1} \longrightarrow \cdots \longrightarrow F_0 \longrightarrow X \longrightarrow 0$ where $F_i$ is flat for every $0 \leq i \leq k$. Then $0 \longrightarrow X^+ \longrightarrow F^+_0  \longrightarrow \cdots \longrightarrow F^+_{k-1} \longrightarrow F^+_k \longrightarrow 0$ is exact since $D^0(\mathbb{Q / Z})$ is an injective complex (notice that ${\rm Ext}^1(Y, D^0(\mathbb{Q / Z})) = 0$ implies that $\overline{\rm Ext}^1(Y, D^0(\mathbb{Q / Z})) = 0$, for every complex $Y$), and $F^+_i$ is an injective complex for every $0 \leq i \leq k$. So ${\rm id}(X^+) \leq k < \infty$ and $X^+ \in \widetilde{\mathcal{W}}$. We have $X^+ \in \widehat{\mathcal{GI}_0} \cap \widetilde{\mathcal{W}} = \widetilde{\mathcal{I}_0}$. It follows $X$ is flat. 
\end{proof} 

The following proposition follows as in Proposition \ref{igualdadesperp}. 

\vspace{0.5cm}

\begin{proposition}\label{cotorsioncomplexeq} $\widehat{\mathcal{GC}} = (\widetilde{\mathcal{F}_0})\mbox{}^\perp \cap \widetilde{\mathcal{W}}$. 
\end{proposition}

\vspace{0.25cm}

From these results, we have that $(\widehat{\mathcal{GF}_0}, (\widetilde{\mathcal{F}_0})\mbox{}^\perp \cap \widetilde{\mathcal{W}})$ and $(\widehat{\mathcal{GF}_0} \cap \widetilde{\mathcal{W}}, (\widetilde{\mathcal{F}_0})\mbox{}^\perp)$ are compatible and complete. The following result follows by Hovey's correspondence.  

\vspace{0.5cm}

\begin{theorem} If $R$ is a commutative Gorenstein ring, then the two compatible and complete cotorsion pairs \[ (\widehat{\mathcal{GF}_0}, (\widetilde{\mathcal{F}_0})\mbox{}^\perp \cap \widetilde{\mathcal{W}}) \mbox{ \ and \ } (\widehat{\mathcal{GF}_0} \cap \widetilde{\mathcal{W}}, (\widetilde{\mathcal{F}_0})\mbox{}^\perp) \] give rise to a unique Abelian model structure on $\Cadl$, called the \underline{Gorenstein-flat model struc}- \underline{ture}, such that $\widehat{\mathcal{GF}_0}$, $(\widetilde{\mathcal{F}_0})^\perp$ and $\widetilde{\mathcal{W}}$ are the classes of cofibrant, fibrant and trivial objects, respectively. 
\end{theorem}

\vspace{0.5cm}

As we did in the Gorenstein-projective and injective cases, the Gorenstein-flat model structures on $\mbox{}_{\mathbb{Z}_4}{\rm {\bf Mod}}$ and ${\rm {\bf Ch}}(\mbox{}_{\mathbb{Z}_4}{\rm {\bf Mod}})$ are not monoidal.

Now it is time to study the Gorenstein-flat dimension. As we did with the class $\mathcal{GP}_0$, it makes sense to consider left $r$-$\mathcal{GF}_0$-modules and left $r$-$\widehat{\mathcal{GF}_0}$-complexes. 

\vspace{0.5cm}

\begin{definition} {\rm A module (resp. complex) is said to be \underline{Gorenstein-$r$-flat} if it is a left $r$-$\mathcal{GF}_0$-module (resp. $r$-$\widehat{\mathcal{GF}_0}$-complex). Let $\mathcal{GF}_r$ (resp. $\widehat{\mathcal{GF}_r}$) denote the class of Gorentein-$r$-flat modules (resp. complexes) and ${\rm Gfd}(X)$ denote the (left) Gorenstein-flat dimension of the module (or complex) $X$. Note $\mathcal{GF}_r = \{ M \in {\rm Ob}(\Modl) \mbox{ : } {\rm Gfd}(M) \leq r \}$ and $\widehat{\mathcal{GF}_r} = \{ X \in {\rm Ob}(\Cadl) \mbox{ : } {\rm Gfd}(X) \leq r \}$. }
\end{definition}

\vspace{0.25cm}

First, we shall prove that the classes $\mathcal{GF}_r$ and $\widehat{\mathcal{GF}_r}$ are the left halves of two complete cotorsion pairs in $\Modl$ and $\Cadl$, respectively. Using this fact we shall construct a model structure on $\Modl$ (resp. $\Cadl$) where $\mathcal{GF}_r$ (resp. $\widehat{\mathcal{GF}_r}$) is the class of cofibrant objects. 

\vspace{0.5cm}

\begin{proposition} Let $R$ be an $n$-Gorenstein ring. Then $(\mathcal{GF}_r, (\mathcal{GF}_r)^\perp)$ is a cotorsion pair for every $0 < r \leq n$. 
\end{proposition}
\begin{proof} We only need to show $\mbox{}^\perp((\mathcal{GF}_r)^\perp) \subseteq \mathcal{GF}_r$. Let $M \in \mbox{}^\perp((\mathcal{GF}_r)^\perp)$. We only need to prove that  ${\rm Tor}^R_{r+1}(W, M) = 0$ for every $W \in \mathcal{W}$, by \cite[Proposition 11.7.5]{EJ}. Using the isomorphism given in \cite[Theorem 3.2.1]{EJ}, we have ${\rm Tor}^R_{r+1}(W, M)^+ \cong {\rm Ext}^{r+1}_R(M, W^+) \cong {\rm Ext}^1_R(M, K)$, where $K \in \Omega^{-r}(W^+)$. We show $K \in (\mathcal{GF}_r)^\perp$. Let $N \in \mathcal{GF}_r$. We have ${\rm Ext}^{1}_R(N,K) \cong {\rm Ext}^{1}_R(G,W^+)$, where $G \in \Omega^r(N)$. Since $N \in \mathcal{GF}_r$, we have $G$ is Gorenstein-flat, and so ${\rm Ext}^1_R(G,W^+) \cong {\rm Tor}^R_1(W, G)^+ = 0$. Hence ${\rm Ext}^1_R(N, K) = 0$ for every $N \in \mathcal{GF}_r$, i.e. $K \in (\mathcal{GF}_r)^\perp$. It follows ${\rm Tor}^R_{r+1}(W, M) = 0$ for every $W \in \mathcal{W}$. 
\end{proof}

Given two resolutions 
\begin{align*}
(\ast) & = (0 \longrightarrow E'_r \longrightarrow E'_{r-1} \longrightarrow \cdots \longrightarrow E'_0 \longrightarrow M' \longrightarrow 0), \mbox{ and} \\
(\ast \ast) & = (0 \longrightarrow E_r \longrightarrow E_{r-1} \longrightarrow \cdots \longrightarrow E_0 \longrightarrow M \longrightarrow 0),
\end{align*}
we shall say that ($\ast$) is a \underline{sub-resolution} of ($\ast \ast$) if ($\ast$) is a sub-complex of ($\ast \ast$).

Recall that for every Gorenstein-flat module $E$ and every $x \in E$, one can construct a $\mathcal{W}$-pure sub-module $S \subseteq E$ with ${\rm Card}(S) \leq \kappa$, such that $x \in S$. One can apply the same reasoning to show that every sub-module $E' \subseteq E$ with ${\rm Card}(E') \leq \kappa$ can be embedded into a $\mathcal{W}$-pure sub-module $S \subseteq E$ with ${\rm Card}(S) \leq \kappa$. From this fact, one deduces the following result. 

\vspace{0.5cm}

\begin{lemma}\label{lemaper} Given $\kappa$ an infinite cardinal satisfying $\kappa > {\rm Card}(R)$, where $R$ is an $n$-Gorenstein ring, let $M \in \mathcal{GF}_r$ with a Gorenstein-flat resolution \[ {\rm (1)} = (0 \longrightarrow E_r \longrightarrow E_{r-1} \longrightarrow \cdots \longrightarrow E_0 \longrightarrow M \longrightarrow 0) \] and $N$ be a sub-module of $M$ with ${\rm Card}(N) \leq \kappa$. Then there exists a Gorenstein-flat sub-resolution \[ (2') = 0 \longrightarrow S'_r \longrightarrow S'_{r-1} \longrightarrow \cdots \longrightarrow S'_0 \longrightarrow N' \longrightarrow 0 \] of {\rm (1)} such that $S'_k$ is a $\mathcal{W}$-pure sub-module of $E_k$ and ${\rm Card}(S'_k) \leq \kappa$, for every $0 \leq k \leq r$, and such that $N \subseteq N'$. Moreover, if $N$ has a sub-resolution \[ (2) = 0 \longrightarrow S_r \longrightarrow S_{r-1} \longrightarrow \cdots \longrightarrow S_0 \longrightarrow N \longrightarrow 0 \] of {\rm (1)}, where $S_k$ is a $\mathcal{W}$-pure sub-module of $E_k$ with ${\rm Card}(S_k) \leq \kappa$, for every $0 \leq k \leq r$, then {\rm (2')} can be constructed in such a way that {\rm (2)} is a sub-resolution of {\rm (2')}. 
\end{lemma}

\vspace{0.25cm}

The previous lemma is proven in \cite[Lemma 5.1]{Perez} for the class $\mathcal{F}_r$ of modules with flat dimension $\leq r$. The same argument works for $\mathcal{GF}_r$. It follows that every Gorenstein-$r$-flat module has a $(\mathcal{GF}_r)^{\leq \kappa}$-filtration. As a consequence of \cite[Proposition 3.2]{Perez}, we have obtained the following theorem.

\vspace{0.5cm}

\begin{theorem} If $R$ is an $n$-Gorenstein ring, then $(\mathcal{GF}_r, (\mathcal{GF}_r)^\perp)$ is a complete cotorsion pair for every $0 < r \leq n$. 
\end{theorem}

\newpage

The pairs $(\mathcal{F}_r, (\mathcal{F}_r)^\perp)$ and $(\mathcal{GF}_r, (\mathcal{GF}_r)^\perp)$ are complete and also compatible. As we did in Corollary \ref{coroproj} and Proposition \ref{igualdadesperp}, we can show the following equalities. 

\vspace{0.25cm}

\begin{proposition} If $R$ is an $n$-Gorenstein ring, then $\mathcal{F}_r = \mathcal{GF}_r \cap \mathcal{W}$ and $(\mathcal{GF}_r)^\perp = (\mathcal{F}_r)^\perp \cap \mathcal{W}$ for every $0 < r \leq n$. 
\end{proposition}

\vspace{0.5cm}

\begin{theorem}\label{gorrflatmodel} If $R$ is an $n$-Gorenstein ring, then the two compatible and complete cotorsion pairs \[ (\mathcal{GF}_r, (\mathcal{F}_r)\mbox{}^\perp \cap \mathcal{W}) \mbox{ \ and \ } (\mathcal{GF}_r \cap \mathcal{W}, (\mathcal{F}_r)^\perp) \] give rise to a unique Abelian model structure on $\Modl$, called the \underline{Gorenstein-$r$-flat model structure}, such that $\mathcal{GF}_r$, $(\mathcal{F}_r)^\perp$ and $\mathcal{W}$ are the classes of cofibrant, fibrant and trivial objects, respectively. 
\end{theorem}

\vspace{0.5cm}

Now we want to apply the same study to the category of complexes $\Cadl$. As we did in the case $r = 0$, suppose that $R$ is a commutative $n$-Gorenstein ring. The following result is a version of \cite[Proposition 11.7.5]{EJ} for the category of chain complexes. 

\vspace{0.5cm} 

\begin{proposition}\label{gorrflatcomplex} The following conditions are equivalent for any chain complex $X \in {\rm Ob}(\Cadl)$ over a commutative $n$-Gorenstein ring $R$:
\begin{itemize}
\item[ {\bf (1)}] $X$ is a Gorenstein-$r$-flat complex.

\item[ {\bf (2)}] $\overline{\rm Tor}_i(W, X) = 0$ for all $i \geq r+1$ and all $W \in \widetilde{\mathcal{W}}$.

\item[ {\bf (3)}] $\overline{\rm Tor}_{r+1}(W, X) = 0$ for all $W \in \widetilde{\mathcal{W}}$.

\item[ {\bf (4)}] Every $r$th Gorenstein-flat syzygy is Gorenstein-flat.

\item[ {\bf (5)}] Every $r$th flat syzygy is Gorenstein-flat.

\item[ {\bf (6)}] $X_m$ is a Gorenstein-$r$-flat module for every $m \in \mathbb{Z}$.

\item[ {\bf (7)}] $X^+$ is a Gorenstein-$r$-injective complex. 
\end{itemize} 
\end{proposition}

\vspace{0.5cm}

\begin{proposition} If $R$ is a commutative $n$-Gorenstein ring, then $(\widehat{\mathcal{GF}_r} \mbox{ }|\mbox{ } (\widehat{\mathcal{GF}_r})\mbox{}^{\overline{\perp}})$ is a bar-cotorsion pair for every $0 < r \leq n$. 
\end{proposition}
\begin{proof} We only need to prove the inclusion $\mbox{}^{\overline{\perp}}((\widehat{\mathcal{GF}_r})\mbox{}^{\overline{\perp}}) \subseteq \widehat{\mathcal{GF}_r}$. Let $X \in \mbox{}^{\overline{\perp}}( ( \widehat{\mathcal{GF}_r})\mbox{}^{\overline{\perp}})$. We want to show that $\overline{\rm Tor}_{r+1}(W, X) = 0$ for every $W \in \widetilde{\mathcal{W}}$. By Proposition \ref{isomorfos}, we have $\overline{{\rm Tor}}_{r+1}(W, X)^+ \cong \overline{\rm Tor}_{r+1}(X, W)^+ \cong \overline{\rm Ext}^{r+1}(X, W^+) \cong \overline{\rm Ext}^1(X, K)$, where $K$ is an $r$th cosyzygy of $W^+$. Let $Z \in \widehat{\mathcal{GF}_r}$. Then $\overline{\rm Ext}^{1}(Z,K) = \overline{\rm Ext}^{r+1}(Z, W^+) \cong \overline{\rm Tor}_{r+1}(W, Z)^+ = 0$, by Proposition \ref{gorrflatcomplex}. So $K \in (\widehat{\mathcal{GF}_r})^{\overline{\perp}}$ and $\overline{\rm Tor}_{r+1}(W, X)^+ \cong \overline{\rm Ext}^1(X, K) = 0$. It follows that $\overline{\rm Tor}_{r+1}(W, X) = 0$ for every $W \in \widetilde{\mathcal{W}}$, i.e. $X \in \widehat{\mathcal{GF}_r}$, by Proposition \ref{gorrflatcomplex} again. 
\end{proof}

Since the class $\widehat{\mathcal{GF}_r}$ is closed under suspensions, we have that $(\widehat{\mathcal{GF}_r}, (\widehat{\mathcal{GF}_r})\mbox{}^{\overline{\perp}})$ is a cotorsion pair and $(\widehat{\mathcal{GF}_r})\mbox{}^\perp = (\widehat{\mathcal{GF}_r})\mbox{}^{\overline{\perp}}$. So $(\widehat{\mathcal{GF}_r}, (\widehat{\mathcal{GF}_r})\mbox{}^\perp)$ is a cotorsion pair. By Lemma \ref{lemaper}, along with the argument used in \cite[Theorem 6.3]{Perez}, we have that $(\widehat{\mathcal{GF}_r}, (\widehat{\mathcal{GF}_r})\mbox{}^\perp)$ is a cotorsion pair cogenerated by $(\widehat{\mathcal{GF}_r})^{\leq \kappa}$, so it is complete. Moreover, it is easy to show that $\widetilde{\mathcal{F}_r} = \widehat{\mathcal{GF}_r} \cap \widetilde{\mathcal{W}}$ and $(\widehat{\mathcal{GF}_r})\mbox{}^\perp = (\widetilde{\mathcal{F}_r})\mbox{}^\perp \cap \widetilde{\mathcal{W}}$. So we obtain the analogous of Theorem \ref{gorrflatmodel} in $\Cadl$, applying Hovey's correspondence. 

\vspace{0.5cm}

\begin{theorem} If $R$ is a commutative $n$-Gorenstein ring, then the two compatible and complete cotorsion pairs \[ (\widehat{\mathcal{GF}_r}, (\widetilde{\mathcal{F}_r})^\perp \cap \widetilde{\mathcal{W}}) \mbox{ \ and \ } (\widehat{\mathcal{GF}_r} \cap \widetilde{\mathcal{W}}, (\widetilde{\mathcal{F}_r})^\perp) \] give rise to a unique Abelian model structure on $\Cadl$, called the \underline{Gorenstein-$r$-flat model struc}- \underline{ture}, such that $\widehat{\mathcal{GF}_r}$, $(\widetilde{\mathcal{F}_r})^\perp$ and $\widetilde{\mathcal{W}}$ are the classes of cofibrant, fibrant and trivial objects, respectively. 
\end{theorem}

\vspace{0.5cm}

We finish this section giving some remarks on covers by the classes $\mathcal{GF}_r$ and $\widehat{\mathcal{GF}_r}$. Recall that given an Abelian category $\mathcal{C}$ and a class $\mathcal{A}$ of objects of $\mathcal{C}$, a map $f : A \longrightarrow X$ is said to be an \underline{$\mathcal{A}$-cover} of $X$ if $A \in \mathcal{A}$ and if the following conditions are satisfied: 
\begin{itemize}
\item[{\bf (1)}] If $f : A' \longrightarrow X$ is another map with $A' \in \mathcal{A}$, then there exists a map $h : A' \longrightarrow A$ (not necessarily unique) such that $f' = f \circ h$.

\item[{\bf (2)}] If $A' = A$ in {\bf (1)}, then every map $h : A \longrightarrow A$ satisfying the equality $f' = f \circ h$ is an automorphism of $A$. 
\end{itemize}
If $f$ satisfies {\bf (1)} but may be not {\bf (2)}, then $f$ is called an \underline{$\mathcal{A}$-pre-cover}. The notions of \underline{$\mathcal{B}$-envelope} and \underline{$\mathcal{B}$-pre-envelope} are dual. A cotorsion pair $(\mathcal{A,B})$ is said to be \underline{perfect} if every object of $\mathcal{C}$ has an $\mathcal{A}$-cover and a $\mathcal{B}$-envelope. In \cite[Corollary 2.3.7]{Gobel}, it is proved that if $(\mathcal{A,B})$ is a complete cotorsion pair and $\mathcal{A}$ is closed under direct limits, then $(\mathcal{A,B})$ is perfect.\footnote{This result is proven for modules over a ring, but it remains true in any Grothendieck category.}

\vspace{0.5cm}

\begin{remark} {\rm \
\begin{itemize}
\item[{\bf (i)}] It is straightforward to show that Gorenstein-flat modules are closed under direct limits. Since $(\mathcal{GF}_0,\mathcal{GC})$ is a complete cotorsion pair with $\mathcal{GF}_0$ closed under direct limits, it follows that $(\mathcal{GF}_0, \mathcal{GC})$ is a perfect cotorsion pair, i.e. every left $R$-module has a Gorenstein-flat cover and a Gorenstein-cotorsion envelope. 

\item[{\bf (ii)}] For every fixed chain complex $X \in {\rm Ob}(\Cadr)$, the functor $\overline{{\rm Tor}}_1(X,-)$ preserves direct limits (This is a consequence of \cite[Proposition 4.2.1 5]{GR}). It follows the class $\widehat{\mathcal{GF}_0}$ is closed under direct limits, and hence the pair $(\widehat{\mathcal{GF}_0}, \widehat{\mathcal{GC}})$ is perfect.  We conclude that every chain complex $X \in \Cadl$ has a Gorenstein-flat cover, provided that $R$ is a commutative $n$-Iwanaga-Gorenstein ring. This result also appears in \cite[Theorem 5.8.4]{GR}, where the proof is more constructive from the author's point of view. 
\end{itemize} }
\end{remark}


\section{Gorenstein homological dimensions over graded rings}

In this section, we shall consider an associative ring $R$ with unit and the $\mathbb{Z}$-graded ring $A:= R[x] / (x^2)$, which has a direct sum decomposition $A = \cdots \oplus 0 \oplus (x) \oplus R \oplus 0 \oplus \cdots$, where the scalars $r \in R$ are the elements of degree $0$, and the elements in the ideal $(x)$ form the terms of degree $-1$. It is not hard to see that the categories $\Modgrad$ and $\Cadl$ are isomorphic. This fact shall allows us to prove that Gorenstein-$r$-projective $A$-modules and dg-$r$-projective chain complexes over $R$ are in one-to-one correspondence, provided $R$ satisfies certain conditions. The same holds for Gorenstein-$r$-injective and Gorenstein-$r$-flat $A$-modules. This result is an extension of the case $r = 0$ proved in \cite[Propositions 3.6, 3.8 and 3.10]{HoveyGillespie}.

Let $\Phi : \Modgrad \longrightarrow \Cadl$ be the functor defined as follows:
\begin{itemize}
\item[ {\bf (1)}] Given a graded $A$-module $M = \bigoplus_{n \in \mathbb{Z}} M_n$, note that if $y \in M_n$ then $x \cdot y \in M_{n-1}$, since $x$ has degree $-1$. Denote by $\Phi(M)_n$ the set $M_n$ endowed with the structure of $R$-module provided by the graded multiplication. Let $\partial_n : \Phi(M)_n \longrightarrow \Phi(M)_{n-1}$ be the map $y \mapsto x \cdot y$. It is easy to check that $\Phi(M) = (\Phi(M)_n, \partial_n)_{n \in \mathbb{Z}}$ is a chain complex over $R$. 

\item[ {\bf (2)}] Let $f : M \longrightarrow N$ be a homomorphism of graded $A$-modules. Then $f(M_n) \subseteq N_n$, for every $n \in \mathbb{Z}$. It follows that $f|_{M_n}$ is an $R$-homomorphism. The family of maps $\{ \Phi(f)_n := f|_{M_n} \}_{n \in \mathbb{Z}}$ defines a chain map $\Phi(f)$. 
\end{itemize}
The functor $\Phi$ is an isomorphism, and we shall denote its inverse by $\Psi : \Cadl \longrightarrow \Modgrad$. This functor establishes the one-to-one correspondences given in \cite[Propositions 3.6, 3.8 and 3.10]{HoveyGillespie}. 

It follows that $\Phi$ and $\Psi$ map projective (resp. injective) objects into projective (resp. injective) objects. It is also easy to check that both $\Psi$ and $\Phi$ preserves exact sequences. 

Concerning flat objects, it is important to recall how tensor products are defined for $\mathbb{Z}$-graded $A$-modules. Given two graded $A$-modules $M = (M_n)_{n \in \mathbb{Z}}$ and $N = (N_m)_{m \in \mathbb{Z}}$, the tensor product $M \otimes_{\mathbb{Z}} N = (\bigoplus_{n + m = k} M_n \otimes_{\mathbb{Z}} N_m)_{k \in \mathbb{Z}}$ has also a $\mathbb{Z}$-graduation. Let $Q$ be the sub-$\mathbb{Z}$-module generated by the elements $(a \cdot y) \otimes z - y \otimes (a \cdot z)$ where $a \in A$, $y \in M$ and $z \in N$. The tensor product of $M$ and $N$ over $A$ is defined by $M \otimes_A N = (M \otimes_{\mathbb{Z}} N) / Q$. It is clear that $M \otimes_A N \cong \Phi(M) \overline{\otimes} \Phi(N)$ and $X \overline{\otimes} Y \cong \Psi(X) \otimes_A \Psi(Y)$ for every $M, N \in {\Modgrad}$ and $X, Y \in \Cadl$. So it follows that $M$ is flat in ${\Modgrad}$ if, and only if, $\Phi(M)$ is flat in $\Cadl$. Similarly, $X$ is flat in $\Cadl$ if, and only if, $\Psi(X)$ is flat in ${\Modgrad}$. The following lemma is straightforward. 

\vspace{0.5cm}

\begin{lemma} \label{isoseqs} For every $i \geq 1$, $M, N \in \Modgrad$, and $Y, Z \in \Cadl$:
\begin{itemize}
\item[ {\bf (1)}] ${\rm Ext}^i_A(M, N) \cong {\rm Ext}^i(\Phi(M), \Phi(N))$.

\item[ {\bf (2)}] ${\rm Tor}^A_i(M, N) \cong \overline{{\rm Tor}}_i(\Phi(M), \Phi(N))$.

\item[ {\bf (3)}] ${\rm Ext}^i(Y, Z) \cong {\rm Ext}^i_A(\Psi(Y), \Psi(Z))$. 

\item[ {\bf (4)}] $\overline{{\rm Tor}}_i(Y, Z) \cong {\rm Tor}^A_i(\Psi(Y), \Psi(Z))$. 
\end{itemize} 
\end{lemma}

\vspace{0.25cm}

Recall from \cite[Definition 3.3]{Gil}, that given a cotorsion pair $(\mathcal{A, B})$ in $\Modl$, a chain complex $X$ is a \underline{dg-$\mathcal{A}$-complex}, denoted $X \in {\rm dg}\widetilde{\mathcal{A}}$, if $X_m \in \mathcal{A}$ for every $m \in \mathbb{Z}$ and if every chain map $X \longrightarrow B$ is null homotopic whenever $B \in \widetilde{\mathcal{B}}$. Similarly, $X$ is a dg-$\mathcal{B}$-complex, denoted $X \in {\rm dg}\widetilde{\mathcal{B}}$, if $X_m \in \mathcal{B}$ for every $m \in \mathbb{Z}$ and if every chain $A \longrightarrow X$ is null homotopic whenever $A \in \widetilde{\mathcal{A}}$. 

The classes ${\rm dg}\widetilde{\mathcal{P}_0}$, ${\rm dg}\widetilde{\mathcal{I}_0}$ and ${\rm dg}\widetilde{\mathcal{F}_0}$ are called dg-projective, dg-injective and dg-flat complexes, respectively. It is known that ${\rm dg}\widetilde{\mathcal{P}_0} = \mbox{}^\perp\mathcal{E}$ and ${\rm dg}\widetilde{\mathcal{I}_0} = \mathcal{E}^\perp$, where $\mathcal{E}$ is the class of exact chain complexes (See \cite[Propositions 2.3.4 and 2.3.5]{GR}). Moreover, $({\rm dg}\widetilde{\mathcal{P}_0}, \mathcal{E})$ and $(\mathcal{E}, {\rm dg}\widetilde{\mathcal{I}_0})$ are cotorsion pairs. 

\vspace{0.5cm}

\begin{corollary} Let $R$ be a left and right Noetherian ring of finite global dimension. Then there is a one-to-one correspondence between the exact chain complexes over $R$ and the $A$-modules in $\mathcal{W}$. 
\end{corollary}
\begin{proof} Let $E$ be an exact complex over $R$. Then ${\rm Ext}^1(X, E) = 0$ for every dg-projective complex $X$. Consider $\Psi(E)$ and let $C$ be a Gorenstein-projective $A$-module. By \cite[Propositions 3.6]{HoveyGillespie} and the comments in the beginning of this section, there exists a unique dg-projective chain complex $X$ such that $C = \Psi(X)$. We have ${\rm Ext}^1_A(C, \Psi(E)) = {\rm Ext}^1_A(\Psi(X), \Psi(E)) \cong {\rm Ext}^1(X, E) = 0$, by Lemma \ref{isoseqs}. It follows $\Psi(E) \in \mathcal{W}$, since $(\mathcal{GP}, \mathcal{W})$ is a cotorsion pair in $\Modgrad$. The mapping $E \mapsto \Psi(E)$ gives the desired correspondence. 
\end{proof}

In \cite[Proposition 3.6]{Gil}, J Gillespie proves that if $(\mathcal{A, B})$ is a complete cotorsion pair in $\Modl$, then $(\widetilde{\mathcal{A}}, {\rm dg}\widetilde{\mathcal{B}})$ and $({\rm dg}\widetilde{\mathcal{A}}, \widetilde{\mathcal{B}})$ are cotorsion pairs in $\Cadl$. Moreover, if $(\mathcal{A, B})$ is also hereditary, then $\widetilde{\mathcal{A}} = {\rm dg}\widetilde{\mathcal{A}} \cap \mathcal{E}$ and $\widetilde{\mathcal{B}} = {\rm dg}\widetilde{\mathcal{B}} \cap \mathcal{E}$. Since the pair $(\mathcal{P}_r, (\mathcal{P}_r)\mbox{}^\perp)$ is complete by \cite[Theorem 7.4.6]{EJ}, it follows that the class ${\rm dg}\widetilde{\mathcal{P}_r}$ of dg-$r$-projective chain complexes is the left half of a cotorsion pair $({\rm dg}\widetilde{\mathcal{P}_r}, (\widetilde{\mathcal{P}_r})\mbox{}^\perp \cap \mathcal{E})$. We have similar results for the classes of dg-$r$-injective chain complexes and dg-$r$-flat complexes. Knowing these facts and using the previous corollary, we can prove the following theorem, which extends the results presented in \cite[Propositions 3.6, 3.8 and 3.10]{HoveyGillespie} to any (Gorenstein) homological dimension. 

\vspace{0.5cm}

\begin{theorem} \label{perezcorresponde} The functor $\Psi : \Cadl \longrightarrow \Modgrad$ maps:
\begin{itemize}
\item[ {\bf (1)}] dg-$r$-projective complexes into Gorenstein-$r$-projective $A$-modules,

\item[ {\bf (2)}] dg-$r$-injective complexes into Gorenstein-$r$-injective $A$-modules, and

\item[ {\bf (3)}] dg-$r$-flat complexes into Gorenstein-$r$-flat $A$-modules.
\end{itemize}
If $R$ is a left and right Noetherian ring of finite global dimension, then the inverse functor $\Phi : \Modgrad \longrightarrow \Cadl$ maps:
\begin{itemize}
\item[ {\bf (1')}] Gorenstein-$r$-projective $A$-modules into dg-$r$-projective complexes,

\item[ {\bf (2')}] Gorenstein-$r$-injective $A$-modules into dg-$r$-injective complexes, and

\item[ {\bf (3')}] Gorenstein-$r$-flat $A$-modules into dg-$r$-flat complexes. 
\end{itemize} 
\end{theorem}
\begin{proof} We only prove {\bf (1)} and {\bf (1')}. Let $X \in {\rm dg}\widetilde{\mathcal{P}_r}$ and consider $\Psi(X)$ along with a partial left projective resolution $0 \longrightarrow C \longrightarrow P_{r-1} \longrightarrow \cdots \longrightarrow P_0 \longrightarrow \Psi(X) \longrightarrow 0$. We show that $C$ is a Gorenstein-projective $A$-module. Consider the complex $\Phi(C)$ and let $E$ be an exact complex. We have ${\rm Ext}^1(\Phi(C), E) \cong {\rm Ext}^1(X, E')$, where $E' \in \Omega^{-r}(E)$. Note that $E' \in (\widetilde{\mathcal{P}_r})\mbox{}^\perp$. In fact, if $Z \in \widetilde{\mathcal{P}_r}$ then ${\rm Ext}^1(Z, E') \cong {\rm Ext}^{r+1}(Z, E) = 0$. Also, it is easy to check that $E' \in \mathcal{E}$. So $E' \in (\widetilde{\mathcal{P}_r})\mbox{}^\perp \cap \mathcal{E} = ({\rm dg}\widetilde{\mathcal{P}_r})\mbox{}^\perp$. It follows ${\rm Ext}^1(\Phi(C), E) \cong {\rm Ext}^1(X, E') = 0$, for every $E \in \mathcal{E}$. In other words, $\Phi(C)$ is dg-projective, and by \cite[Propositions 3.6]{HoveyGillespie} we have that $C = \Psi(\Phi(C))$ is a Gorenstein-projective $A$-module. 

Now suppose that $R$ is a left and right Noetherian ring of finite global dimension. Note that $\Psi$ and $\Phi$ also define an one-to-one correspondence between $r$-projective complexes over $R$ and $r$-projective $A$-modules. Let $X \in (\widetilde{\mathcal{P}_r})^\perp$ and consider $\Psi(X)$. Let $M$ be an $r$-projective $A$-module. Then $\Phi(M)$ is an $r$-projective complex. We have ${\rm Ext}^1_A(M, \Psi(X)) \cong {\rm Ext}^1(\Phi(M), X) = 0$. It follows $\Psi(X) \in (\mathcal{P}_r)^\perp$. Hence, $\Psi$ and $\Phi$ give rise to a one-to-one correspondence between $(\widetilde{\mathcal{P}_r})\mbox{}^\perp$ and $(\mathcal{P}_r)\mbox{}^\perp$. Also, by the previous corollary, we have also have a one-to-one correspondence between $\mathcal{E}$ and $\mathcal{W}$. Since $({\rm dg}\widetilde{\mathcal{P}_r})\mbox{}^\perp = (\widetilde{\mathcal{P}_r})\mbox{}^\perp \cap \mathcal{E}$ and $(\mathcal{P}_r)\mbox{}^\perp \cap \mathcal{W} = (\mathcal{GP}_r)^\perp$, we have that a complex $Y$ is in $({\rm dg}\widetilde{\mathcal{P}_r})\mbox{}^\perp$ if and only if $\Psi(Y)$ is in $(\mathcal{GP}_r)^\perp$. Since ${\rm dg}\widetilde{\mathcal{P}_r} = \mbox{}^\perp( ({\rm dg}\widetilde{\mathcal{P}_r})\mbox{}^\perp)$ and $\mathcal{GP}_r = \mbox{}^\perp( (\mathcal{GP}_r)^\perp )$, we have that $\Phi$ maps Gorenstein-$r$-projective $A$-modules into dg-$r$-projective complexes.
\end{proof}


\section*{Acknowledgements} 

The author would like to thank Professor Andr\'e Joyal of the Department of Mathematics at UQ\`AM for patiently listening to my ideas and explanations on this work. This research has been financially supported by ISM-CIRGET and Fondation de l'UQ\`AM. 



\begin{thebibliography}{99}

\addcontentsline{toc}{section}{References}

\bibitem{Aldrich} Aldrich, S. T.; Enochs, E. E.; Garc\'ia Rozas, J. R.; \& Oyonarte, L. {\it Covers and Envelopes in Grothendieck Categories: Flat Covers of Complexes with Applications}. Journal of Algebra. Vol. 243. pp. 615-630.
(2001).

\bibitem{Rada} Bravo, D.; Enochs, E. E.; Iacob, A.; Jenda, O. M. G.; \& Rada, J. {\it Cotorsion pairs in C($R$-Mod)}. Rocky Mountain Journal of Mathematics. Vol. 42, No. 6, 1787-1802.  (2012).

\bibitem{Eklof} Eklof, P.; \& Trlifaj, J. {\it How to make Ext vanish}. Bulletin of the London Mathematical Society, Vol. 33, No. 1, pp. 41-51. (2001). 

\bibitem{EJ} Enochs, E. E.; \& Jenda, O. M. G. {\it Relative Homological Algebra. Volume 1}. De Gruyter Expositions in Mathematics, Vol. 30. De Gruyter. Berlin. (2011).  

\bibitem{Estrada} Enochs, E. E.; Estrada, S.; \& Garc\'ia Rozas, J. R. {\it Gorenstein categories and Tate cohomology on projective schemes}. Mathematische Nachrichten. Vol. 281, no. 4., pp. 525-540. (2008).

\bibitem{GR} Garc\'ia Rozas, J. R. {\it Covers and envelopes in the category of complexes of modules}. Cahpman \& Hall/CRC Research Notes in Mathematics. Vol. 407. London (1999). 

\bibitem{Gillespie} Gillespie, J. {\it Cotorsion Pairs and Degreewise Homological Model Structures}. Homology, Homotopy and Applications, Vol. 10, No. 1, pp. 283-304. (2008)

\bibitem{Gil} Gillespie, J. {\it The Flat Model Structure on Ch(R)}. Transactions of the American Mathematical Society, Vol. 356, No. 8, pp. 3369-3390. (2004). 

\bibitem{HoveyGillespie} Gillespie, J.; \& Mark Hovey. {\it Gorenstein model structures and generalized derived categories}. Submitted to Proceedings of the Edinburgh Mathematical Society.

\bibitem{Gobel} G\"obel, R.; \& Trlifaj, J. {\it Approximations and Endomorphism Algebras of Modules}. De Gruyter Expositions in Mathematics, Vol. 41. Walter de Gruyter. Berlin (2006).

\bibitem{Grothendieck} Alexander, G. {\it Sur quelques points d'alg\`ebre homologique.} Tohoku Mathematical Journal. Vol. 9, No. 2, pp. 119-221. (1957).

\bibitem{Hovey} Hovey, M. {\it Cotorsion pairs, model category structures and representation theory}. Mathematical Zeitschrift. Vol. 241, pp. 553-592 (2002).

\bibitem{HoveyBook} Hovey, M. {\it Model categories}. Mathematical Surveys and Monographs. Volume 63. American Mathematical Society. Providence (2007). 

\bibitem{Osborne} Osborne, M. S. {\it Basic Homological Algebra}. Graduate Texts in Mathematics, Vol. 196. Springer Verlag. New York, New York, United States. (2000).

\bibitem{Pareigis} Pareigis, B. {\it Categories and Functors}. Pure and Applied Mathematics, Vol. 39. Academic Press. New York. (1970).

\bibitem{Perez} P\'erez, M. {\it Homological dimensions and Abelian model structures on chain complexes}. Preprint: http://arxiv.org/abs/1202.2481.

\end{thebibliography}
\end{document}